\documentclass{article}
\usepackage{color, xcolor, colortbl}
\usepackage{graphicx,epstopdf}
\usepackage{geometry}
\usepackage{enumerate}
\usepackage{amsmath,amssymb,amsthm}
\usepackage{algorithm}
\usepackage{algorithmic}
\usepackage{caption}

\usepackage{bm}
\usepackage{appendix}
\usepackage{multirow}
\usepackage{braket}
\usepackage[english]{babel}
\usepackage{hyperref}
\usepackage[capitalize]{cleveref}
\usepackage[sort&compress,square,numbers]{natbib}
\usepackage{adjustbox}
\usepackage{xspace}
\usepackage[roman]{complexity}
\usepackage{soul}
\usepackage{tikz}
\usetikzlibrary{quantikz}
\usepackage{rotating}
\usepackage{setspace}
\usepackage{fancyhdr}
\usepackage{float}
\usepackage{subfig}
\usepackage{authblk}

\newcommand{\norm}[1]{\left\lVert#1\right\rVert}

\newcommand{\RR}{\mathbb{R}}

\renewcommand{\Re}{\operatorname{Re}}

\newcommand{\I}{\mathrm{i}}

\newtheorem{thm}{\protect\theoremname}
\theoremstyle{plain}
\newtheorem{lemma}[thm]{\protect\lemmaname}
\theoremstyle{plain}
\newtheorem{rem}[thm]{\protect\remarkname}
\theoremstyle{plain}
\theoremstyle{plain}

\theoremstyle{plain}
\newtheorem{cor}[thm]{\protect\corollaryname}

\providecommand{\definitionname}{Definition}
\providecommand{\assumptionname}{Assumption}
\providecommand{\corollaryname}{Corollary}
\providecommand{\lemmaname}{Lemma}
\providecommand{\propositionname}{Proposition}
\providecommand{\remarkname}{Remark}
\providecommand{\theoremname}{Theorem}

\title{Splitting Analysis for Yukawa Potential}

\author[1,2]{Di Fang}
\author[1,2]{Jiaqi Zhang}

\affil[1]{Department of Mathematics, Duke University}
\affil[2]{Duke Quantum Center, Duke University}

\date{} 

\begin{document}
\maketitle

\begin{abstract}
Splitting methods are among the most classical and fundamental tools for the simulation of quantum dynamics, and their importance has grown further with the rise of quantum computing.
In this work, we analyze the Schr\"odinger equation with Yukawa potential, a physically relevant and widely used model potential.
It may be viewed as a Coulomb interaction with exponential decay at spatial infinity, preserving the Coulomb singularity at the origin while removing the long-range Coulomb tail.
We prove that the operator splitting for this unbounded Hamiltonian achieves a global $1/4$-order convergence rate in the time step for many-body Yukawa interactions, with explicit polynomial dependence on the number of particles.
The result holds for all initial wavefunctions in $H^2(\mathbb R^{3N})$, the natural domain of the Hamiltonian, and our numerical experiments are consistent with the theoretical estimates.
To identify the sharp obstruction behind this rate, we prove a short-time lower bound in the one-body setting of order $t^{5/4}$ for the one-step error, which rules out any uniform global estimate of order better than $1/4$ in general.
This agreement with the optimal $1/4$ rate in the Coulomb case is particularly interesting, as Yukawa potential is short-ranged compared to Coulomb potential.
For the many-body upper bound, one of the new technical ingredients is the explicit polynomial-in-system-size Sobolev estimates of many-body Yukawa systems. These estimates are crucial for obtaining fully a priori bounds that depend only on the norms of the initial states, rather than on the solution at time $t$. For the one-body lower bound, we leverage a new analysis argument based on Fourier analysis and Kato smoothing.
\end{abstract}
\tableofcontents

\section{Introduction}

Splitting algorithms have enjoyed a long history, beginning with early theoretical work on Lie product formulas and Trotter-type semigroup approximations~\cite{Trotter1959}, later introduced as numerical algorithms by Strang~\cite{Strang1968}, and subsequently developing into a widely studied class of numerical methods in the numerical analysis community (see, e.g., books and reviews~\cite{McLachlanQuispel2002,Lubich2008book,BaoCai2013,LasserLubich2020,JinMarkowichSparber2011,BlanesCasasMurua2008,BlanesCasasMurua2024} and the papers cited therein). More recently, they have gained renewed prominence due to their central role in quantum simulation and quantum computing. Indeed, the simulation of quantum systems was one of the original motivations for the development of quantum computers~\cite{Feynman1982}, and Hamiltonian simulation remains one of the most basic and fundamental problems in the field. From this perspective, a central objective is to obtain rigorous and provable performance guarantees for algorithms applied to physically relevant systems. Among the many available approaches, the Lie-Trotter splitting retains a particularly important role because of its structural simplicity, its direct implementability, and its strong performance in a wide range of applications.
In many previous rigorous analyses of quantum simulation and quantum computing algorithms in general, the most favorable settings are those in which the Hamiltonian terms are bounded~\cite{Lloyd1996,BerryAhokasCleveEtAl2007,BerryChilds2012,BerryChildsCleveEtAl2014,BerryChildsCleveEtAl2015,BerryChildsKothari2015,LowChuang2017,ChildsMaslovNamEtAl2018,LowWiebe2019,Low2019,ChildsSu2019,ChildsSuTranEtAl2020,AnFangLin2022}. In the Schr\"odinger setting, this often means assuming that the potential is bounded or sufficiently regular, together with suitable bounds on its derivatives. Under such assumptions, analysis and complexity estimates are much more accessible~\cite{Somma2015,SahinogluSomma2020,AnFangLin2021,SuHuangCampbell2021,ZhaoZhouShawEtAk2021,ChildsLengEtAl2022,FangTres2023,BornsWeilFang2022,HuangTongFangSu2023,ZengSunJiangZhao2022,GongZhouLi2023,LowSuTongTran2023,ZhaoZhouChilds2024,YuXuZhao2024,ChenXuZhaoYuan2024,FangQu2025,BeckerGalkeSalzmannLuijk2024,MizutaIkeda2025Fujii2025,MizutaKuwahara2025}, and in certain time-dependent formulations one may even observe enhanced phenomena such as superconvergence
\cite{FangLiuSarkar2025,BornsweilFangZhang2025,FangLiuZhu2025,FangZhang2025}.

From the perspective of quantum simulation, the primary regime of interest is the many-body regime, where quantum efficiency requires error bounds with explicit polynomial dependence in the system size. Existing many-body Hamiltonian simulation analysis can often handle bounded Hamiltonians or sufficiently smooth interaction potentials through operator norm commutator estimates. The situation becomes substantially more delicate once one moves to Schr\"odinger operators with unbounded kinetic energy and singular interaction potentials. In that regime, the standard bounded operator commutator analysis is no longer directly applicable, since the Laplacian is unbounded on $L^2$, where wave function lives in $L^2$, and singular potentials make the commutator structure much more delicate. As a result, one must work on the domain of the Hamiltonian, rather than on the whole Hilbert space, and the analysis must be carried out in stronger norms that capture the regularity required for the equation to make sense. This difficulty is closely related to a broader theme in the analysis of unbounded Hamiltonian simulation, where one often needs state dependent or domain level estimates rather than crude operator norm bounds on $L^2$ \cite{AnFangLin2021, FangWuSoffer2025,KivlichanWiebeBabbushEtAl2017,TongAlbertMccleanPreskillSu2022,ZhengLengLiuWu2024,BeckerGalkeSalzmannLuijk2024,BurgarthGalkeHahnvanLuijk2023}. 

On the other hand, splitting methods for Schr\"odinger-type equations have a long and celebrated history in numerical analysis. Much of this literature focuses on one-body or few-body problems, including low-regularity nonlinear systems, and is primarily concerned with obtaining sharp convergence rates under minimal regularity assumptions; see, for instance, \cite{BaoCai2013,BaoWang2024,BaoWang2025,BaoWang2024b,BaoMaWang2024}. However, these works are usually not designed to track the dependence of the constants explicitly in the many-body system size. The goal of our work is to bridge the quantum simulation and numerical analysis perspectives by proving sharp splitting error estimates for many-body Schr\"odinger dynamics with singular Yukawa interactions, with constants that are explicit and polynomial in the system size.

In this paper, we study the Schr\"odinger equation with the Yukawa potential, a widely used screened interaction model arising in charge-stabilized colloids, Yukawa fluids, complex (dusty) plasmas, dark matter phenomenology, and screened hybrid electronic structure calculations for solids \cite{Rowlinson1989,KremerRobbinsGrest1986,RobbinsKremerGrest1988,NaidooSchnitker1994,OhtaHamaguchi2000,LiuGoreeVaulina2006,LoebWeiner2011,TranBlaha2011}. 
The Yukawa potential may be viewed as a screened Coulomb interaction: 
it retains the Coulomb-type singularity at the origin while exhibiting exponential decay at large distances.
Analytically, the Yukawa potential occupies a particularly interesting position. As $|x|\to 0$, it behaves like the Coulomb potential $1/|x|$, so the singularity at the origin remains fully present. As $|x|\to\infty$, however, the exponential factor dominates, and hence the potential is short-ranged. Thus the Yukawa potential is still unbounded and still singular, but it no longer has the long Coulomb tail. In this sense, compared with Coulomb, it retains one major difficulty while removing another.

This perspective is also closely related to recent rigorous work on the Coulomb case. In \cite{FangWuSoffer2025,FangWu2026}, the splitting errors for many-body quantum dynamics with Coulomb interactions were analyzed without smoothing or regularizing the singular potential, consistent with the physical and numerical discoveries in \cite{BurgarthFacchiHahnJohnssonYuasa2024}.
From the present point of view, the Coulomb interaction combines three major difficulties at once: it is unbounded, it is singular at the origin, and it is long-ranged. These features make the commutator analysis especially delicate and strongly influence the final convergence behavior.
The Yukawa potential may be viewed as a screened version of Coulomb potential, and this makes it a particularly natural next model to study.

This difference suggests a clear guiding intuition for the present work. Since the singularity at the origin is still present, one should not expect the Yukawa problem to become entirely routine. On the other hand, because the interaction decays exponentially at spatial infinity, the Yukawa potential separates the local Coulomb-type singularity from the long-range Coulomb tail. The central issue is therefore whether the leading splitting behavior is governed primarily by the singularity at the origin, or whether the exponential screening at infinity can improve the convergence rate. 
This leads to an important question:
\begin{center}
    \textit{
    Question 1: Can one rigorously quantify the error explicitly in the system size of the splitting algorithm for the Schr\"odinger equation with Yukawa interactions in the many-body case?
    }
\end{center}
This system size dependence is crucial from the perspective of quantum simulation: polynomial dependence on the system size is compatible with quantum efficiency, whereas exponential dependence on the system size would make many-body quantum simulation inefficient.
At the same time, from the numerical-analysis perspective, one would like to know whether the short-range behavior induced by exponential screening can improve the leading splitting convergence rate. If not, one should further ask whether this rate can be proved sharp by establishing a matching lower bound. This leads to a second question:
\begin{center}
 \textit{
Question 2: Does exponential screening improve the leading splitting rate for the Yukawa potential, and if not, can one rigorously prove sharpness through a matching lower bound?}
\end{center}
Both questions are addressed in this work. We first answer Question 1 affirmatively. We prove a rigorous many-body long-time error bound for the first-order splitting algorithm applied to Schr\"odinger equations with Yukawa interactions. The estimate holds for arbitrary initial wavefunctions in $H^2(\RR^{3N})$, the natural domain of the many-body Hamiltonian, rather than for specially chosen states or smaller invariant subspaces, and its prefactor depends polynomially on the particle number $N$. 
Our answer to Question 2 is twofold: exponential screening does not improve the leading splitting convergence rate, and a matching lower bound establishes its sharpness. Indeed, the short-range Yukawa interaction still has a Coulomb-type singularity at the origin, which dictates the leading obstruction. More precisely, the many-body upper bound gives a global $Tt^{1/4}$ error estimate, while the one-body lower bound gives a matching one-step obstruction of order $t^{5/4}$.

The main technical novelty in the upper bound analysis is the passage from one-body estimates to many-body estimates with explicit dependence on the system size. If one does not track the constants in $N$, the many-body argument is not straightforward but can typically be generalized by adapting one-body estimates. The essential difficulty in the present many-body setting is instead to keep this dependence explicit and polynomial in $N$. In particular, estimates depending on Sobolev norms of $\psi(t)$ are not sufficient for our purpose, since $\psi(t)$ is precisely the evolved state we aim to approximate and its relevant norms are not known a priori. We therefore need estimates depending only on the initial state $\psi(0)$, with constants explicit in the system size. The central estimate that we establish for many-body Yukawa Hamiltonians is such a polynomial in $N$ Sobolev flow bound, which is one of our main contributions and may be of independent interest elsewhere. In order to achieve this, we leverage the Hardy-Littlewood-Sobolev structure of the pairwise Yukawa interaction.
For the lower bound, we isolate the leading contribution through Fourier analysis, while controlling the other contributions by duality together with Kato smoothing, among other analysis techniques. This proves that the $t^{5/4}$ one-step obstruction is sharp, suggesting that no global rate better than $1/4$ can hold in general for any final time $T$.

\subsection{Problem Setup}\label{sect:pb_setup}

We work on the Hilbert space $L^2(\RR^{3N})$, where $N\ge1$ is the number of particles. We denote the particle positions by
$x=(x_1,\ldots,x_N)$, with $x_j\in\RR^3$.
Let $\mu>0$ be a fixed constant, and define the Hamiltonian in terms of the kinetic operator and potential: $H_\mu=-\Delta + V_{\mu}$, where 
\begin{equation}\label{eq:yukawa_nbody_potential}
\begin{aligned}
-\Delta=-\sum_{j=1}^N \Delta_{x_j},
&\quad
V_\mu(x)
=\sum_{1\le j<k\le N}
c_{jk}
\frac{e^{-\mu |x_j-x_k|}}{|x_j-x_k|} \quad \mbox{for} \quad N\geq2, \\
& \mbox{and} \quad V_\mu(x)
=-\frac{e^{-\mu |x|}}{|x|} \quad \mbox{for}\quad N=1.
\end{aligned}
\end{equation}
Here, $c_{jk}\in \RR, 1\leq j<k\leq N$, and we denote the maximal pair-interaction coefficient by
$C_{\mathrm{int}}
:=\max_{1\le j<k\le N}|c_{jk}|<\infty$.
We consider the Schr\"odinger evolution
\begin{equation}\label{eq:yukawa_nbody_schrodinger}
\begin{cases}
i\partial_t\psi(t)=H_\mu\psi(t),\\[0.3em]
\psi(0)=\psi_0\in H^2(\RR^{3N}).
\end{cases}
\quad t\in\RR .
\end{equation}
where the solution can be expressed as $\psi(t)=e^{-itH_\mu}\psi_0$.
Since $0\le e^{-\mu|y|}/|y|\le |y|^{-1}$ for $y\in\RR^3\setminus\{0\}$, the Yukawa interaction is no more singular than the Coulomb interaction at the
origin. By the Kato-Rellich theorem
\cite[Theorem X.12]{ReedSimon1975}; see also \cite{CyconFroeseKirschSimon1987}, $V_\mu$ is infinitesimally $-\Delta$-bounded. Hence $H_\mu$ is self-adjoint on $H^2(\RR^{3N})$. In particular, the domain of $H_{\mu}$ is $H^2$.
We use the following $H^2$ norm on $\RR^{3N}$:
\begin{equation}\label{eq:H2_norm_yukawa_nbody}
\|g\|_{H^2}
:=\left(
\|(-\Delta)g\|_{L^2(\RR^{3N})}^2
+\|g\|_{L^2(\RR^{3N})}^2
\right)^{1/2}.
\end{equation}
Throughout, $C>0$ denotes a universal constant that may change from line to line, as is typical in mathematical and numerical estimates, and subscripts indicate dependence on the displayed parameters. 
For a final time $T>0$, we divide the interval $[0,T]$ into $L$ time steps of size $t=T/L$, and set $t_j=jt$ for $j=0,\ldots,L$. On each subinterval $[t_j,t_{j+1}]$, we apply the celebrated first-order splitting algorithm, also known as the Lie-Trotter formula \cite{Trotter1959}, which is given over one step of size $t$ by
\begin{equation}\label{eq:yukawa_nbody_trotter}
U_{1,\mu}(t)
:=e^{-itV_\mu}e^{it\Delta}.
\end{equation}
So overall, the full $L$-step approximation is
$U_{1,\mu}(t)^L
=\left(e^{-itV_\mu}e^{it\Delta}\right)^L$.

\subsection{Main Results and Understanding}
\label{sec:main-results}

Our first main result is the many-body long-time upper bound for the first-order
Lie-Trotter splitting. It shows that the global error converges with order
$1/4$ in the time step and has explicit polynomial dependence on the particle number $N$.

\begin{thm}[Many-body long-time Yukawa Trotter upper bound]
\label{thm:main-many-body-yukawa-upper}
Let $N\ge 2$, and let $H_\mu=-\Delta+V_\mu$ be the many-body Yukawa Hamiltonian, with $-\Delta$ and $V_\mu$ defined in \cref{eq:yukawa_nbody_potential}.
Assume that $C_{\mathrm{int}}=\max_{1\le j<k\le N}|c_{jk}|<\infty$, and fix $\mu>0$.
Let $T>0$, $L\in\mathbb N$, and set
$t=T/L\in(0,1]$. Then, for every
$\psi_0\in H^2(\mathbb R^{3N})$, the first-order splitting operator
$U_{1,\mu}(t)$ defined in \cref{eq:yukawa_nbody_trotter} satisfies
\begin{equation}
\left\|
U_{1,\mu}(t)^L\psi_0
-e^{-iTH_\mu}\psi_0
\right\|_{L^2(\mathbb R^{3N})}
\le
\widetilde C_{\mu,C_{\mathrm{int}}}\,
N^{9/2}\,
T\,t^{1/4}
\|\psi_0\|_{H^2(\mathbb R^{3N})},
\label{eq:main-many-body-yukawa-upper}
\end{equation}
where $\widetilde C_{\mu,C_{\mathrm{int}}}>0$ depends only on $\mu$ and
$C_{\mathrm{int}}$, and is independent of $t$, $T$, $L$, $N$, and $\psi_0$.
\end{thm}
Here, we emphasize several aspects of \cref{thm:main-many-body-yukawa-upper}. First, the estimate has explicit polynomial dependence on $N$, despite the Hilbert space is exponential in $N$. 
Second, the theorem gives a global convergence rate of $1/4$ in the time step, the same leading rate as in the Coulomb case. This raises the natural question of whether the rate is merely an artifact of the proof or is genuinely sharp; below, we support sharpness numerically and prove it analytically through a one-step lower bound of order $t^{5/4}$, which rules out any uniform global rate better than $1/4$.
Moreover, we allow arbitrary initial data in $H^2(\RR^{3N})$, the domain of the Hamiltonian, meaning that the theorem applies to all quantum states for which the action of the Hamiltonian is well defined in $L^2$. Finally, our bounds are fully a priori, depending only on the initial data $\psi_0$ rather than on the unknown evolved state $\psi(t)$. It would be much easier to obtain an intermediate estimate depending on Sobolev norms of the unknown evolved state $\psi(t)$, which is precisely the state we aim to approximate. In that case, the $N$-dependence can be dramatically improved to $N^{3/2}$. Thus, a crucial novel ingredient is an explicit system-size Sobolev norm estimate for the many-body Yukawa potential; see \cref{prop:yukawa-H2-flow}.

\begin{thm}[Short-time Yukawa lower bound]
\label{thm:main-one-body-yukawa-lower}
Fix $\mu>0$. Let $H_\mu=-\Delta+V_\mu$ be the Yukawa Hamiltonian defined by the $N=1$ case of \cref{eq:yukawa_nbody_potential}, 
and let $U_{1,\mu}(t)$ be defined in \cref{eq:yukawa_nbody_trotter}. 
Assume that $H_\mu$ admits a ground state $\psi_\ast\in H^2(\mathbb R^3)$ with $\psi_\ast(0)\neq0$.
Then there exist constants $c_{\psi_\ast}>0$ and $t_{0,\mu,\psi_\ast}>0$ such that, for $0<t<t_{0,\mu,\psi_\ast}$,
\begin{equation}
\left\|
U_{1,\mu}(t)\psi_\ast
-e^{-itH_\mu}\psi_\ast
\right\|_{L^2(\mathbb R^3)}
\ge
c_{\psi_\ast}t^{5/4}.
\label{eq:main-one-body-yukawa-lower}
\end{equation}
\end{thm}
We remark that, although \cref{thm:main-one-body-yukawa-lower} may appear to require the a priori existence of a ground state, the Yukawa Hamiltonian is known to possess a ground state for $\mu<\mu_{\mathrm c}$; see, for example, \cite{GaravelliOliveira1991,EdwardsEtAl2017}. In the normalization used here, the critical screening parameter is $\mu_{\mathrm c}\approx 0.595$. Note that under this scenario, the ground state naturally has the property $\psi_\ast(0)\neq 0$. In addition, we remark that the worst case is saturated by the ground state, which is completely physically. 

\begin{rem}[General coupling case]
\label{rem:general-coupling-one-body-lower}
The preceding \cref{thm:main-one-body-yukawa-lower} is stated for the attractive Yukawa potential and
ground state initial data. An alternative statement for
\begin{equation}\label{eq:general_Ham}
H_{\mu,Z}:=-\Delta+Z\frac{e^{-\mu|x|}}{|x|},
\quad Z\in\mathbb R \setminus \{0\},
\end{equation}
is given in \cref{thm:general-one-body-yukawa-lower} in \cref{app:general-coupling-lower}. In that formulation,
a ground state is not assumed; instead, the lower bound is proved for a
suitable fixed test function
$\psi_\ast\in C_c^\infty(\mathbb R^3)$ with $\psi_\ast(0)\neq0$.
\end{rem}

\begin{cor}[no uniform rate better than $1/4$]\label[corollary]{cor:no-better-rate}
Let $H_\mu$ and $U_{1,\mu}(t)$ and $\psi_\ast$ be as in \cref{thm:main-one-body-yukawa-lower}. There do not exist constants $C>0$ and $\alpha>1/4$ such that
\begin{equation}\label{eq:forbidden-global-bound}
\norm{U_{1,\mu}(t)^L-e^{-iTH_\mu}}_{H^2\to L^2}
\le C T t^\alpha,
\qquad T>0,\quad t=T/L\in(0,1],
\end{equation}
holds for all integers $L\ge 1$.
\end{cor}
\begin{rem}
The lower bound argument is local in time and proves the precise one-step obstruction %
$\left\|
U_{1,\mu}(s)-e^{-is H_\mu}
\right\|_{H^2(\mathbb R^3)\to L^2(\mathbb R^3)}
\gtrsim s^{5/4}$
and thereby rules out every uniform global estimate with exponent strictly larger than $1/4$. It does \emph{not} by itself prove a fixed-$T$ lower bound for $\norm{U_{1,\mu}(t)^L-e^{-iTH_\mu}}_{H^2\to L^2}$ as $L\to\infty$.
\end{rem}

\section{Proof Overview}
\label{subsec:proof-roadmap}
\cref{fig:proof-roadmap} summarizes the logical structure of
the proof. For the upper bound, a
telescoping reduction and a pairwise commutator estimate yield a local $t^{5/4}$ bound, which is closed by the many-body Sobolev flow estimate. For the lower bound, the exact three-term decomposition isolates the leading delta contribution and two lower-order non-delta terms.
\begin{figure}[H]
\centering
\begin{adjustbox}{max width=\textwidth,max totalheight=0.82\textheight}
\begin{tikzpicture}[
    node distance=0.45cm and 0.55cm,
    roadbox/.style={
        rectangle,
        rounded corners=2pt,
        draw=black!70,
        very thin,
        align=center,
        inner xsep=4pt,
        inner ysep=3pt,
        font=\scriptsize,
        text width=3.25cm
    },
    titlebox/.style={
        roadbox,
        fill=gray!12,
        draw=black!70,
        font=\scriptsize\bfseries,
        text width=3.75cm
    },
    upperbox/.style={
        roadbox,
        fill=blue!4,
        draw=blue!60!black
    },
    lowerbox/.style={
        roadbox,
        fill=orange!5,
        draw=orange!70!black
    },
    deltabox/.style={
        roadbox,
        fill=orange!10,
        draw=orange!85!black
    },
    nonbox/.style={
        roadbox,
        fill=purple!5,
        draw=purple!65!black
    },
    finalbox/.style={
        roadbox,
        fill=green!6,
        draw=green!55!black,
        font=\scriptsize\bfseries
    },
    roadarrow/.style={
        -{Latex[length=1.7mm]},
        thick,
        draw=black!70
    },
    dasharrow/.style={
        -{Latex[length=1.7mm]},
        thick,
        dashed,
        dash pattern=on 2.8pt off 1.8pt,
        draw=black!70,
        shorten <=1pt,
        shorten >=1pt
    }
]

\node[upperbox, text width=3.2cm] (uglobal) at (-4.1,-1.4)
{\cref{sec:yukawa-global-to-local}\\
global $\Rightarrow$ local};

\node[titlebox, above=0.55cm of uglobal, text width=3.7cm] (utitle)
{Many-body upper bound};

\node[upperbox, below left=1.25cm and -0.72cm of uglobal, text width=2.15cm] (ucomm)
{\cref{lem:pairwise-yukawa-commutator}\\
\textbf{cutoff balance}\\
$s^{1/4}$};

\node[upperbox, below=0.55cm of ucomm, text width=2.15cm] (ulocal)
{\cref{eq:yukawa-local-error-sup-flow}\\
local error in\\
Sobolev norms\\
$t^{5/4}$};

\node[upperbox, right=1.45cm of ucomm, text width=2.25cm] (uflow)
{\cref{prop:yukawa-H2-flow}\\
\textbf{HLS + Sobolev flow}\\
$N$-dependence};

\node[
    finalbox,
    below=1.15cm of ulocal.south,
    anchor=north,
    xshift=1.80cm,
    text width=3.10cm
] (utheorem)
{\cref{thm:main-many-body-yukawa-upper}\\
sum over $L=T/t$\\
$O(Tt^{1/4})$};

\draw[dasharrow] (uglobal.south west) -- (ucomm.north);
\draw[roadarrow] (ucomm.south) -- (ulocal.north);
\draw[dasharrow] (ucomm.east) -- (uflow.west);

\draw[roadarrow] (ulocal.south) -- ([xshift=-0.55cm]utheorem.north);
\draw[roadarrow] (uflow.south) -- ([xshift=0.55cm]utheorem.north);

\node[
    draw=blue!35,
    rounded corners=3pt,
    fit=(utitle)(uglobal)(ucomm)(ulocal)(uflow)(utheorem),
] (upperframe) {};

\node[titlebox] (ltitle) at (4.10,0.15)
{One-body lower bound};

\node[lowerbox, below=0.55cm of ltitle, text width=3.83cm] (ldecomp)
{\cref{eq:lower_error_three_term_decomposition}\\
\textbf{exact error $+$ Duhamel}\\ delta term $+$ \\two non-delta terms};

\node[deltabox, below left=0.58cm and -0.28cm of ldecomp, text width=2.7cm] (ldelta)
{\cref{thm:exact_delta_contribution_t54}\\
\textbf{Fourier analysis}\\
leading order $t^{5/4}$};

\node[nonbox, below=0.58cm of ldecomp, text width=2.45cm] (lzero)
{\cref{lem:exterior_zero_order}\\
\textbf{H\"older}\\
lower order};

\node[nonbox, below right=0.58cm and -0.31cm of ldecomp, text width=2.75cm] (lgrad)
{\cref{lem:exterior_first_order}\\
\textbf{duality + Kato}\\
lower order};

\node[deltabox, below=0.395cm of ldelta, text width=2.7cm] (lphase)
{\cref{lem:replace_delta_coefficient_lower_bound,lem:replace_yukawa_phase_delta_lower_bound}\\
coefficient $+$ phase replacements};

\node[lowerbox, below=1.70cm of lzero, text width=3.65cm] (lcombine)
{combine pieces\\
\textbf{reverse triangle inequality}};

\node[finalbox, below=0.55cm of lcombine, text width=3.65cm] (lthm)
{\cref{thm:main-one-body-yukawa-lower}\\
one-step lower bound\\
$t^{5/4}$};

\node[finalbox, below=0.45cm of lthm, text width=3.65cm] (lcor)
{\cref{cor:no-better-rate}\\
no rate better than $1/4$};

\draw[dasharrow] (ldecomp.south west) -- (ldelta.north);
\draw[dasharrow] (ldecomp.south) -- (lzero.north);
\draw[dasharrow] (ldecomp.south east) -- (lgrad.north);
\draw[roadarrow] (ldelta) -- (lphase);
\draw[roadarrow] (lphase) -- (lcombine);
\draw[roadarrow] (lzero) -- (lcombine);
\draw[roadarrow] (lgrad) -- (lcombine);
\draw[roadarrow] (lcombine) -- (lthm);
\draw[roadarrow] (lthm) -- (lcor);

\node[
    draw=orange!45,
    rounded corners=3pt,
    fit=(ltitle)(ldecomp)(ldelta)(lzero)(lgrad)(lphase)(lcombine)(lthm)(lcor),
    inner sep=3.6pt
] (lowerframe) {};

\end{tikzpicture}
\end{adjustbox}

\caption{Proof roadmap. Solid arrows indicate steps that lead to the subsequent result, while dashed arrows indicate auxiliary estimates that suffice to establish the corresponding step.
}
\label{fig:proof-roadmap}
\end{figure}
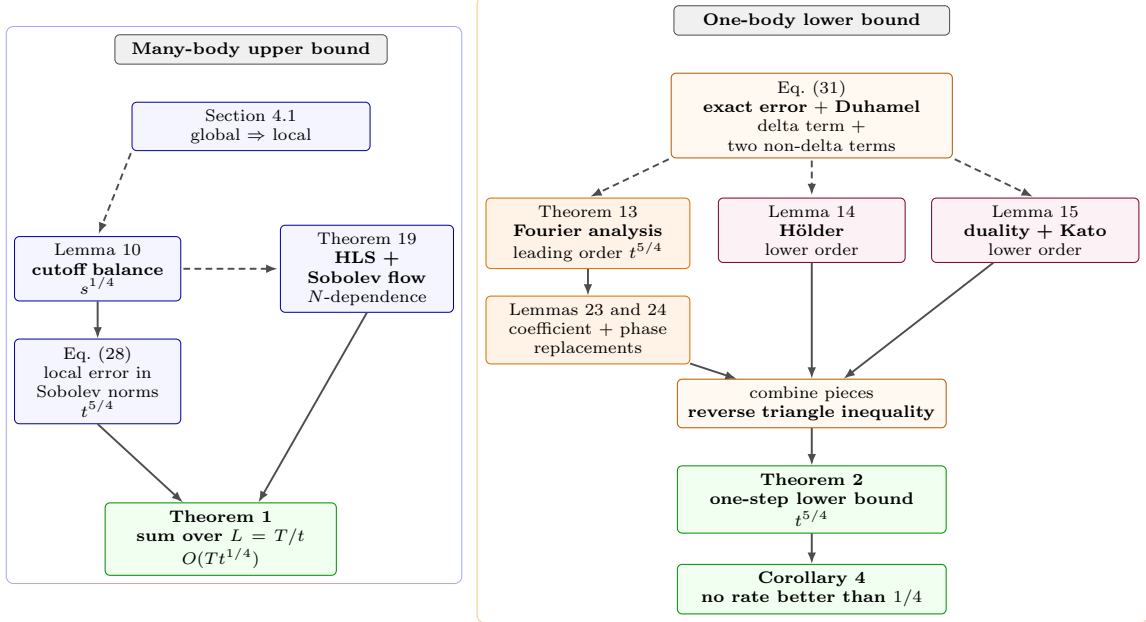

\section{Technical preliminaries}
\label{sect:tech_prelim}

In this section, we present several preliminary results that will be used
repeatedly in the proofs of the main results. 
The first is a standard duality characterization, where its proof follows the usual duality argument, see for example \cite[Theorem~4.4]{RudinFA}.
\begin{lemma}[Dual characterization of the operator norm]
\label[lemma]{lem:op_norm_dual}
Let $X$ be a normed space, let $Y$ be a Hilbert space, and let
$A:X\to Y$ be a bounded linear operator. Then
\begin{equation}
\label{eq:op_norm_dual}
\|A\|_{\mathrm{op}}
=\|A\|_{X\to Y}
=\sup_{\substack{\|x\|_X\le 1\\ \|y\|_Y\le 1}}
\left|
\left\langle y,Ax\right\rangle_Y
\right|.
\end{equation}
\end{lemma}

We also recall the free Kato smoothing estimate used in the lower bound proof:
\begin{lemma}[Kato Smoothing]\label[lemma]{lem:kato}
For every $\psi_0\in L^2(\RR^3)$,
\begin{equation}
\left\| |x|^{-1} e^{it\Delta} \psi_0 \right\|_{L_t^2L_x^2}
\lesssim
\|\psi_0\|_{L_x^2}.
\end{equation}
\end{lemma}
The above Kato smoothing follows from \cite[Theorem~1 and Remark~(c), in particular (1.4)]{KatoYajima1989}.
Throughout this paper, we use this free Kato smoothing estimate, corresponding to the special case $n=3, \sigma=1, \gamma=0$, and $H=-\Delta$ of the general smoothing framework based on the sharp dispersive estimate in \cite[(1.6)]{MizutaniYao2021}. With this choice, \cite[(1.6)]{MizutaniYao2021} yields the Kato smoothing estimate with weight $|x|^{-1}$ for the Schr\"odinger group $e^{it\Delta}$. In Kato's original paper \cite{KatoYajima1989}, this lemma corresponds to the case of supersmoothing.

\begin{lemma}[Hardy-Littlewood-Sobolev (HLS) operator bound]\label[lemma]{lem:hls_operator}
Let $p=-i\nabla$ on $L^2(\RR^3)$, and let $|p|=(-\Delta)^{1/2}$. Then $|p|^{-1}|x|^{-1}$ and $|x|^{-1}|p|^{-1}$ extend to bounded operators on $L^2(\RR^3)$, with
\begin{equation}
\bigl\||p|^{-1}|x|^{-1}\bigr\|_{L^2\to L^2}
=\bigl\||x|^{-1}|p|^{-1}\bigr\|_{L^2\to L^2}
=2.
\end{equation}
\end{lemma}
For the above HLS bound, see \cite[Theorem~2.5]{Herbst1977}; see also \cite{FangWuSoffer2025}.

Finally, we recall the exact local error representation that connects the splitting error to the commutator estimates used below.
Exact error representations for splitting methods are a celebrated topic, and there are various forms. In particular, some standard forms may place the singular potential as the rightmost factor in an operator product or fully expand the commutator using only the Hamiltonian summands
\cite{ChildsSuTranEtAl2020, AnFangLin2021, DescombesThalhammer2010, JahnkeLubich2000}, so the required regularity need not propagate through the expression. Although these forms are equivalent when all operators are bounded, domain issues may arise for unbounded operators as in our case. In particular, if $V$ appears as the rightmost factor, then it acts first and may fail to map the relevant domain into itself. We therefore choose a particular equivalent exact error representation, already derived in \cite[Lemma~9]{FangWuSoffer2025}, since for unbounded singular potentials this form is compatible with the required domain-to-domain mapping.
Namely, if $H=A+B$ and $U_1(t)=e^{-itB}e^{-itA}$,
then
\begin{equation}\label{eq:local_error_representation_intro}
U_1(t)-e^{-itH}
=i\int_0^t
e^{-isB}
\,[e^{-isA},B]\,
e^{-i(t-s)H}\,ds .
\end{equation}
Therefore, the local estimates below reduce to commutator bounds for $[e^{is\Delta},V_\mu]$.

\section{Proof of the Upper Bound}
\label{sec:upper-bound}

In this section, we prove the many-body long-time upper bound stated in
\cref{thm:main-many-body-yukawa-upper}. We first reduce the global error to a sum of local error terms by a standard telescoping argument (\cref{sec:yukawa-global-to-local}). Using the first-order local error representation, each local term is written as a single time integral involving the commutator $[e^{is\Delta},V_\mu]$ acting on the exact solution at an intermediate time.
This reduces the proof to two estimates: a pairwise Yukawa commutator estimate (\cref{lem:pairwise-yukawa-commutator}) and as a crucial ingredient, the system-size-explicit many-body Sobolev flow estimate for the exact Yukawa dynamics in \cref{prop:yukawa-H2-flow}. The many-body bookkeeping follows the Coulomb analysis in \cite[Sec.~4 and Sec.~5]{FangWuSoffer2025}. For simplicity of discussion, throughout this section we write $\|\cdot\|$ for the usual $L^2(\RR^{3N})$ norm, or for the induced $L^2(\RR^{3N})\to L^2(\RR^{3N})$ operator norm, depending on the context. Any other norm will be specified explicitly.

\subsection{Global-to-local reduction}
\label{sec:yukawa-global-to-local}
We use the time step notation introduced in \cref{sect:pb_setup}.
It suffices to prove the estimate in the regime $0<t\le1$. 
For $t\in[0,T]$, let $U_\mu(t):=e^{-itH_\mu}$, $\psi(t)=U_\mu(t)\psi_0=e^{-itH_\mu}\psi_0$, and recall that $U_{1,\mu}$ is defined in \cref{eq:yukawa_nbody_trotter}.
The first-order splitting operator with step size $t$ is $U_{1,\mu}(t)$ defined in \cref{eq:yukawa_nbody_trotter}.
Since $T=Lt$, we have $U_\mu(T)=U_\mu(t)^L$. By the standard
telescoping identity,
\begin{equation}
U_{1,\mu}(t)^L-U_\mu(T)
=\sum_{j=0}^{L-1}
U_{1,\mu}(t)^{L-1-j}
\left(
U_{1,\mu}(t)-U_\mu(t)
\right)
U_\mu(t_j).
\label{eq:yukawa-global-telescoping}
\end{equation}
Since both $U_{1,\mu}(t)$ and $U_\mu(t)$ are unitary on $L^2$ (the self-adjointness of $H_\mu$ is justified below in \cref{prop:yukawa-H2-flow}), applying
\cref{eq:yukawa-global-telescoping} to $\psi_0$ gives
\begin{equation}
\left\|
U_{1,\mu}(t)^L\psi_0
-U_\mu(T)\psi_0
\right\|
\le
\sum_{j=0}^{L-1}
\left\|
\left(
U_{1,\mu}(t)-U_\mu(t)
\right)
U_\mu(t_j)\psi_0
\right\|                           
=
\sum_{j=0}^{L-1}
\left\|
\left(
U_{1,\mu}(t)-U_\mu(t)
\right)
\psi(t_j)
\right\|.
\label{eq:yukawa-global-error-reduced}
\end{equation}
Thus the global error is reduced to estimating the local error estimates
\begin{equation} %
\left(
U_{1,\mu}(t)-U_\mu(t)
\right)
\psi(t_j)
=i\int_0^t
e^{-isV_\mu}
\left[
e^{is\Delta},V_\mu
\right]
\psi(t_j+t-s)\,ds,
\label{eq:yukawa-local-error-at-tj}
\end{equation}
where we used exact error representation in \cref{eq:local_error_representation_intro}.
As is clear from \cref{eq:yukawa-local-error-at-tj}, the proof now requires two system-size-explicit estimates. First, we need an $H^2\to L^2$ commutator bound for $[e^{is\Delta},V_\mu]$, obtained by summing the pairwise Yukawa estimates while keeping the dependence on $s$, $\mu$, $C_{\mathrm{int}}$, and $N$ explicit. Second, we need an a priori many-body $H^2$ flow estimate for the exact Yukawa dynamics, controlling the Sobolev norms of $\psi(t_j+t-s)$ by $\|\psi_0\|_{H^2}$ with polynomial dependence on the system size.

\subsection{$N$-body local Trotter error estimate}
\label{sec:nbody-trotter-error}

In this subsection, we reduce each local error term to Sobolev norms of the exact
Yukawa solution at intermediate times. This is the local estimate needed to complete the many-body long-time upper bound in \cref{thm:main-many-body-yukawa-upper}.
For $0<\beta<1$, set $a=s^\beta$, and let
$v_{\mu,\mathrm{reg}}$ and $v_{\mu,\mathrm{sin}}$ be the cutoff pieces
defined in \cref{eq:yukawa-pair-cutoff-decomposition}. 
The proof of \cref{lem:pairwise-yukawa-commutator} relies on the following cutoff estimates \cref{lem:yukawa-cutoff-estimates}, which we state here for later reference and establish its proof in \cref{app:pairwise-yukawa-commutator}.

\begin{lemma}[Cutoff estimates for Yukawa]
\label[lemma]{lem:yukawa-cutoff-estimates}
Fix $\mu>0$, then there exist constants
$C_{\Delta,\mu}>0$, $C_{\nabla,\mu}>0$, and $C_{\mathrm{sin}}>0$,
depending only on $\mu$ and on the fixed cutoff function, such that for all $0<s\le1$,
\begin{align}
\left\|
\Delta_y v_{\mu,\mathrm{reg}}(\cdot,s)
\right\|_{L^2(\mathbb R^3)}
&\le
C_{\Delta,\mu}a^{-3/2}
=C_{\Delta,\mu}s^{-3\beta/2},
\label{eq:yukawa-cutoff-laplacian-bound}
\\
\left\|
|y|\nabla_y v_{\mu,\mathrm{reg}}(\cdot,s)
\right\|_{L^\infty(\mathbb R^3)}
&\le
C_{\nabla,\mu}a^{-1}
=C_{\nabla,\mu}s^{-\beta},
\label{eq:yukawa-cutoff-gradient-bound}
\\
\left\|
v_{\mu,\mathrm{sin}}(\cdot,s)
\right\|_{L^2(\mathbb R^3)}
&\le
C_{\mathrm{sin}}s^{\beta/2}.
\label{eq:yukawa-cutoff-singular-bound}
\end{align}
\end{lemma}
\begin{lemma}[Pairwise Yukawa commutator estimate]\label[lemma]{lem:pairwise-yukawa-commutator}
Fix $\mu>0$, and for $1\le j<k\le N$ set $V_{\mu,jk}(x):=e^{-\mu|x_j-x_k|}/|x_j-x_k|$.
There exists a constant
$\widetilde C_{F,\mu}>0$, depending only on $\mu$ and on the fixed cutoff
function, such that for all
$0<s\le1$, $1\le j<k\le N$, and $f\in H^2(\RR^{3N})$,
\begin{equation}
\left\|
[e^{is\Delta},V_{\mu,jk}]f
\right\|_{L^2(\mathbb R^{3N})}
\le
\widetilde C_{F,\mu}s^{1/4}
\left(
\|\langle p_j\rangle^2 f\|
+\|\langle p_k\rangle^2 f\|
\right),
\label{eq:pairwise-yukawa-commutator}
\end{equation}
where, for $\ell=j,k$, $p_\ell=-i\nabla_{x_\ell}$ and
$\langle p_\ell\rangle^2=1+|p_\ell|^2=1-\Delta_{x_\ell}$.
\end{lemma}
\begin{proof}
By symmetry, it suffices to prove the estimate for the pair $(j,k)=(1,2)$. Write $V_{\mu,12}(x)=V_\mu(x_1-x_2)$, where $V_\mu(y)=e^{-\mu|y|}/|y|$.
Thus the singular variable for this pair is the relative coordinate
$y=x_1-x_2$. All remaining variables, including the center-type coordinate
associated with $x_1,x_2$ and the variables $x_3,\ldots,x_N$, are passive for
this pairwise multiplication operator.
Using \cref{eq:yukawa-pair-cutoff-decomposition}, decompose
$V_{\mu,12}=V_{\mu,\mathrm{reg}}(x_1-x_2,s)+V_{\mu,\mathrm{sin}}(x_1-x_2,s)$.
We first estimate the regular part. Since $V_{\mu,\mathrm{reg}}$ depends only on $x_1-x_2$, we have
$[\Delta,V_{\mu,\mathrm{reg}}(x_1-x_2,s)]
=2[\Delta_{x_1-x_2},V_{\mu,\mathrm{reg}}(x_1-x_2,s)]$. Applying Duhamel's commutator formula and writing $f_u=e^{i(s-u)\Delta}f$ and $y=x_1-x_2$, we get
\begin{equation}
\left\|
[e^{is\Delta},V_{\mu,\mathrm{reg}}]f
\right\|
\le
2\int_0^s
\left\|
(\Delta_y V_{\mu,\mathrm{reg}})(x_1-x_2,s)f_u
\right\|\,du                         
+
4\sum_{\ell=1}^3
\int_0^s
\left\|
(\partial_{y_\ell}V_{\mu,\mathrm{reg}})(x_1-x_2,s)
\partial_{y_\ell}f_u
\right\|\,du .
\label{eq:yukawa-reg-comm-start}
\end{equation}
For the zeroth-order term, we use \cref{eq:yukawa-cutoff-laplacian-bound} in \cref{lem:yukawa-cutoff-estimates} together with the Sobolev embedding in the $x_1$ variable,
$\|g\|_{L^\infty_{x_1}}\le C_{\mathrm{Sob}}\|\langle p_1\rangle^2 g\|_{L^2_{x_1}}$.
Since $e^{i(s-u)\Delta}$ commutes with $\langle p_1\rangle^2$, this gives $\left\|
(\Delta_y V_{\mu,\mathrm{reg}})(x_1-x_2,s)f_u
\right\|
\le
C_{\mu}s^{-3\beta/2}
\|\langle p_1\rangle^2 f\|$.
Therefore,
\begin{equation}
2\int_0^s
\left\|
(\Delta_y V_{\mu,\mathrm{reg}})(x_1-x_2,s)f_u
\right\|\,du
\le
C_{\mu}s^{1-3\beta/2}
\|\langle p_1\rangle^2 f\|.
\label{eq:yukawa-reg-zero-bound}
\end{equation}
For the first-order term, using
\cref{eq:yukawa-cutoff-gradient-bound} in \cref{lem:yukawa-cutoff-estimates}, the HLS bound in the $x_1$ variable,
the commutation of the free propagator with $p_1$ and $\partial_{y_\ell}$, and
\cref{lem:yukawa-relative-derivative-estimate}, we obtain
\begin{equation}
\begin{aligned}
4\sum_{\ell=1}^3
\int_0^s
\left\|
(\partial_{y_\ell}V_{\mu,\mathrm{reg}})(x_1-x_2,s)
\partial_{y_\ell}f_u
\right\|\,du
&\le
C_{\mu}s^{1-\beta}
\left(
\|\langle p_1\rangle^2 f\|
+
\|\langle p_2\rangle^2 f\|
\right)  \\
&\le
C_{\mu}s^{1-3\beta/2}
\left(
\|\langle p_1\rangle^2 f\|
+
\|\langle p_2\rangle^2 f\|
\right),
\label{eq:yukawa-reg-first-bound}
\end{aligned}
\end{equation}
where the last inequality uses $0<s\le1$. Combining this with
\cref{eq:yukawa-reg-zero-bound} gives
\begin{equation}
\left\|
[e^{is\Delta},V_{\mu,\mathrm{reg}}]f
\right\|
\le
C_{\mu}s^{1-3\beta/2}
\left(
\|\langle p_1\rangle^2 f\|
+
\|\langle p_2\rangle^2 f\|
\right).
\label{eq:yukawa-reg-final-bound}
\end{equation}
We now estimate the singular part. By unitarity of $e^{is\Delta}$ on $L^2$,
$\|[e^{is\Delta},V_{\mu,\mathrm{sin}}]f\|
\le
\|V_{\mu,\mathrm{sin}}(x_1-x_2,s)f\|
+\|V_{\mu,\mathrm{sin}}(x_1-x_2,s)e^{is\Delta}f\|$.
Using \cref{eq:yukawa-cutoff-singular-bound} in \cref{lem:yukawa-cutoff-estimates} and Sobolev embedding in the $x_1$ variable, we obtain
\begin{equation}
\left\|
[e^{is\Delta},V_{\mu,\mathrm{sin}}]f
\right\|
\le
C s^{\beta/2}
\|\langle p_1\rangle^2 f\|
\le
C s^{\beta/2}
\left(
\|\langle p_1\rangle^2 f\|
+\|\langle p_2\rangle^2 f\|
\right).
\label{eq:yukawa-singular-final-bound}
\end{equation}
Combining \cref{eq:yukawa-reg-final-bound,eq:yukawa-singular-final-bound} and choosing
$\beta=\frac12$, for which
$1-\frac32\beta=\frac{\beta}{2}=\frac14$, gives \cref{eq:pairwise-yukawa-commutator} for the pair $(1,2)$. Relabeling the variables proves the result for every $1\le j<k\le N$.
\end{proof}

The above \cref{lem:pairwise-yukawa-commutator} is for the $N\geq 2$ case, and for the $N=1$ case, we can do a similar analysis; we present the full description in \cref{app:one_upper}.
\begin{lemma}[One-body Yukawa commutator estimate]
\label[lemma]{lem:one-body-yukawa-commutator-upper}
Fix $\mu>0$, and let $V_\mu$ be the one-body Yukawa potential given by the $N=1$ case of \cref{eq:yukawa_nbody_potential}. Then, for $0<s\le1$,
\begin{equation}
\left\|
[e^{is\Delta},V_\mu]
\right\|_{H^2(\mathbb R^3)\to L^2(\mathbb R^3)}
\le C_{\mu}s^{1/4}.
\label{eq:one-body-yukawa-commutator-upper}
\end{equation}
\end{lemma}

\subsection{Proof of \cref{thm:main-many-body-yukawa-upper}}
\label{sec:yukawa-completion-upper-bound}
We now combine the local error representation, the pairwise commutator estimate,
and the many-body Sobolev flow estimate. Fix $j=0,\ldots,L-1$.
Using the local error representation \cref{eq:yukawa-local-error-at-tj} and the
unitarity of $e^{-isV_\mu}$ on $L^2$, we obtain
\begin{equation}
\begin{aligned}
\left\|
\left(
U_{1,\mu}(t)-U_\mu(t)
\right)
\psi(t_j)
\right\|                               \le
\int_0^t
\left\|
\left[
e^{is\Delta},V_\mu
\right]
\psi(t_j+t-s)
\right\|\,ds .
\end{aligned}
\end{equation}
Since
$V_\mu=\sum_{1\le a<b\le N}c_{ab}V_{\mu,ab}$,
we have
$[e^{is\Delta},V_\mu]
=\sum_{1\le a<b\le N}
c_{ab}[e^{is\Delta},V_{\mu,ab}]$.
Therefore, by the triangle inequality and
\cref{lem:pairwise-yukawa-commutator},
\begin{equation}
\begin{aligned}
\left\|
\left[
e^{is\Delta},V_\mu
\right]
\psi(t_j+t-s)
\right\|                               \le
\widetilde C_{F,\mu} C_{\mathrm{int}} s^{1/4}
\sum_{1\le a<b\le N}
\left(
\|\langle p_a\rangle^2\psi(t_j+t-s)\|
+\|\langle p_b\rangle^2\psi(t_j+t-s)\|
\right).
\end{aligned}
\end{equation}
After substituting this estimate into the preceding integral, we take the
supremum over $t\in[t_j,t_{j+1}]$, since
$t_j+t-s\in[t_j,t_{j+1}]$, and then integrate in $s$ to obtain
\begin{equation}
\left\|
\left(
U_{1,\mu}(t)-U_\mu(t)
\right)
\psi(t_j)
\right\|                                   \le
\frac{4}{5}
\widetilde C_{F,\mu} C_{\mathrm{int}} t^{5/4}
\sup_{t\in[t_j,t_{j+1}]}
\sum_{1\le a<b\le N}
\left(
\|\langle p_a\rangle^2\psi(t)\|
+\|\langle p_b\rangle^2\psi(t)\|
\right).
\label{eq:yukawa-local-error-sup-flow}
\end{equation}
Combining \cref{eq:yukawa-local-error-sup-flow} with the momentum summation bound and the $H^2$ flow estimate in \cref{eq:yukawa-sobolev-sum-flow}, we have
$\|(U_{1,\mu}(t)-U_\mu(t))\psi(t_j)\|
\le
C_{\mu,C_{\mathrm{int}}}N^{9/2}t^{5/4}\|\psi_0\|_{H^2}$
for every $j=0,\ldots,L-1$.
Substituting this into the telescoping bound \cref{eq:yukawa-global-error-reduced} and using $L=T/t$, we obtain
\begin{equation}
\left\|
U_{1,\mu}(t)^L\psi_0
-e^{-iTH_\mu}\psi_0
\right\|
\le
C_{\mu,C_{\mathrm{int}}}
N^{9/2}Tt^{1/4}
\|\psi_0\|_{H^2}.
\end{equation}
This proves \cref{thm:main-many-body-yukawa-upper}.

As an easier intuition for the main mechanism, and for readers interested only in the one-body problem, we include in \cref{app:one_upper} a self-contained proof of the corresponding upper bound.

\section{Proof of the Lower Bound}\label{sec:lower_bound}
In this section, we prove the one-body matching lower bound, which identifies the $t^{5/4}$ one-step obstruction and shows that the global $1/4$ rate is sharp. Throughout this section, we work in the one-body attractive Yukawa case, namely the $N=1$ case of \cref{eq:yukawa_nbody_potential}.
The commutator expansion and the distributional Yukawa identity are
\begin{equation}
\label{eq:two_non_delta}
[\Delta,V_\mu]g
=(\Delta V_\mu)g+2\nabla V_\mu\cdot\nabla g,
\quad
\Delta V_\mu
=\mu^2V_\mu+4\pi\delta_0 .
\end{equation}
The sign of the delta mass is fixed by the attractive convention in \cref{eq:yukawa_nbody_potential}.

By the exact local error representation \cref{eq:yukawa-local-error-at-tj} and the standard Duhamel formula for the
commutator, together with \cref{eq:two_non_delta}, we obtain, for $\sigma\ge0$,
\begin{equation}
\label{eq:lower_error_three_term_decomposition}
\begin{aligned}
&\left(U_{1,\mu}(t)-e^{-\I tH_\mu}\right)
e^{-\I\sigma H_\mu}\psi_0    \\
&\quad =
-\int_0^t\int_0^s
e^{-\I sV_\mu}
e^{\I\tau\Delta}
\Bigl(
4\pi g_{t,s,\sigma,\tau}(0)\delta_0
+\mu^2V_\mu g_{t,s,\sigma,\tau}
+2\nabla V_\mu\cdot\nabla g_{t,s,\sigma,\tau}
\Bigr)
\,d\tau\,ds,
\end{aligned}
\end{equation}
where $g_{t,s,\sigma,\tau}:=
e^{\I(s-\tau)\Delta}
e^{-\I(t-s+\sigma)H_\mu}\psi_0$.
The three summands in
\cref{eq:lower_error_three_term_decomposition} are, respectively, the singular delta contribution and the two non-delta contributions. We treat the delta contribution in \cref{sec:lower-bound}, and the two non-delta contributions in \cref{sec:lower-bound-nondelta}.

\subsection{Lower Bound for the Delta Contribution} 
\label{sec:lower-bound}
We now focus on the delta contribution in
\cref{eq:lower_error_three_term_decomposition}, which is the singular leading term. In the first summand, the singular source is $4\pi\delta_0 g_{t,s,\sigma,\tau}$. By Sobolev embedding $H^2(\RR^3)\hookrightarrow C^0(\RR^3)$, we have
$\delta_0g_{t,s,\sigma,\tau}
=g_{t,s,\sigma,\tau}(0)\delta_0$ in the sense of distributions. Thus the exact delta contribution is
\begin{equation}
\label{eq:exact_delta_contribution_sigma}
-4\pi
\int_0^t\int_0^s
e^{-\I sV_\mu}
e^{\I\tau\Delta}
\left(
g_{t,s,\sigma,\tau}(0)\delta_0
\right)
\,d\tau\,ds .
\end{equation}
The next lemma computes the free time-integrated delta term which remains after the coefficient and phase replacements.

\begin{lemma}[Scaling of the free delta contribution]
\label[lemma]{lem:free_delta_contribution_scaling}
Let
\begin{equation}
\label{eq:F_delta_definition_lower}
F_\Delta(t,x)
:=\int_0^t ds
\int_0^s d\tau\,
e^{\I\tau\Delta}\delta_0(x).
\end{equation}
Then $F_\Delta(t,\cdot)\in L^2(\mathbb R^3)$ for every $t>0$, and there
exists a constant $C_\Delta>0$, independent of $t$, such that
for every $g_0\in\mathbb C$,
\begin{equation}
\label{eq:delta_core_with_coefficient_lower}
\left\|
4\pi g_0F_\Delta(t,\cdot)
\right\|_{L^2(\mathbb R^3)}
=4\pi |g_0|C_\Delta t^{5/4}.
\end{equation}
\end{lemma}
\begin{proof}
All expressions involving $e^{\I\tau\Delta}\delta_0$ are first understood in
the sense of tempered distributions. The $L^2$ statement is then proved from
the Fourier representation of the time integrated distribution.
Let
$u(\tau,x):=e^{\I\tau\Delta}\delta_0(x)$.
Then $u$ is the free Schr\"odinger evolution with delta initial data:
\begin{equation}
\label{eq:free_sch_delta_pde}
\I\partial_\tau u=-\Delta u, \quad u(0)=\delta_0.
\end{equation}
Equivalently, $u$ is the fundamental solution of the free Schr\"odinger
equation. For general initial data $f$, one has
$e^{\I\tau\Delta}f=K_\tau*f$,
where $K_\tau(x)
=e^{\I|x|^2/(4\tau)}/(4\pi\I\tau)^{3/2}$ is understood as the oscillatory kernel associated with the Fourier multiplier $e^{-\I\tau|k|^2}$.
In particular, $e^{\I\tau\Delta}\delta_0=K_\tau*\delta_0=K_\tau$.
Thus $F_\Delta(t,x)=\int_0^t\int_0^sK_\tau(x)\,d\tau\,ds$.
For the $L^2$ estimate, it is more convenient to work on the Fourier side.
We use the Fourier transform convention
\begin{equation} \label{eq:fourier_convention}
\widehat f(k)
=\int_{\mathbb R^3}e^{-\I x\cdot k}f(x)\,dx,
\quad
f(x)
=\frac{1}{(2\pi)^3}
\int_{\mathbb R^3}e^{\I x\cdot k}\widehat f(k)\,dk.
\end{equation}
By Plancherel's theorem, equivalently Parseval's identity,
$\|f\|_{L_x^2}^2
=\|\widehat f\|_{L_k^2}^2/(2\pi)^3$.
Since
$\widehat{\delta_0}(k)=1$, and $\widehat{e^{\I\tau\Delta}\delta_0}(k)
=e^{-\I\tau |k|^2}$,
we obtain
$\widehat{F_\Delta}(t,k)
=\int_0^t ds
\int_0^s d\tau\,
e^{-\I\tau |k|^2}$.
Changing the order of integration over the triangle $0\le \tau\le s\le t$, we have
$\widehat{F_\Delta}(t,k)
=\int_0^t
(t-\tau)e^{-\I\tau |k|^2}
\,d\tau$.
We set $\tau=t\theta, \
0\le\theta\le1$, then
$\widehat{F_\Delta}(t,k)
=t^2
\int_0^1
(1-\theta)e^{-\I\theta |k|^2t}
\,d\theta$.
By defining
$h(q)
:=\int_0^1
(1-\theta)e^{-\I\theta |q|^2}
\,d\theta, \
q\in\RR^3$,
we have
\begin{equation}
\label{eq:F_delta_hat_h_k_sqrt_t_lower}
\widehat{F_\Delta}(t,k)
=t^2h(k\sqrt t).
\end{equation}
Thus the relevant variable in $h$ is not $k$ itself, but
$\widetilde{k}:=k\sqrt t$.
We next show that $h$ has a finite nonzero $L^2$ norm. First,
$|h(q)|
\le\int_0^1(1-\theta)\,d\theta
=\frac{1}{2}$.
For large $|q|$, write $a=|q|^2$. Direct integration gives
$h(q)=1/(\I a)-
(1-e^{-\I a})/(\I a)^2,
\quad
a=|q|^2$,
where the value at $a=0$ is understood by continuity. Hence, for $|q|\ge1$, $|h(q)| \le \frac{C}{|q|^2}$.
Therefore
$\int_{\RR^3}|h(q)|^2\,dq
\le C+C\int_{|q|\ge1}|q|^{-4}\,dq<\infty$.
Moreover,
$h(0)
=\int_0^1(1-\theta)\,d\theta
=\frac{1}{2}$,
so $h\not\equiv0$. Thus
$0<\|h\|_{L^2(\mathbb R^3)}<\infty$.
Using Plancherel's theorem and \cref{eq:F_delta_hat_h_k_sqrt_t_lower}, we get
\begin{equation}
\label{eq:F_delta_parseval_scaling_lower}
\|F_\Delta(t,\cdot)\|_{L_x^2}^2
=\frac{1}{(2\pi)^3}
\int_{\mathbb R^3}
\left|
t^2h(k\sqrt t)
\right|^2
\,dk 
=\frac{t^4}{(2\pi)^3}
\int_{\mathbb R^3}
|h(k\sqrt t)|^2
\,dk.
\end{equation}
With the change of variables $\widetilde{k}=k\sqrt t$, and since $k\in\mathbb R^3$, we have %
$dk=t^{-3/2}\,d\widetilde{k}$.
Therefore,
\begin{equation}
\label{eq:h_scaling_explicit_lower}
\int_{\mathbb R^3}|h(k\sqrt t)|^2\,dk
=t^{-3/2}
\int_{\mathbb R^3}|h(\widetilde{k})|^2\,d\widetilde{k}
=t^{-3/2}\|h\|_{L^2(\mathbb R^3)}^2.
\end{equation}
Substituting this into \cref{eq:F_delta_parseval_scaling_lower}, and taking the square root gives
$\|F_\Delta(t,\cdot)\|_{L_x^2}
=(2\pi)^{-3/2}
\|h\|_{L^2(\mathbb R^3)}
t^{5/4}
=: C_\Delta t^{5/4}$.
Multiplying by $4\pi g_0$ gives
\cref{eq:delta_core_with_coefficient_lower}.
\end{proof}
We now return to the exact delta contribution in
\cref{eq:exact_delta_contribution_sigma}. Suppose for the moment that the two delta replacements have been justified. The first is the coefficient replacement, where in the final time-integrated $L^2$ norm, the coefficient $g_{t,s,\sigma,\tau}(0)$ is replaced by its limiting ground state value $e^{-\I E_0\sigma}\psi_\ast(0)$. The second is the phase replacement, where the left Yukawa phase $e^{-\I sV_\mu}$ is removed from the leading delta contribution. After these two replacements, the remaining leading object is exactly the free time-integrated delta term computed in \cref{lem:free_delta_contribution_scaling}.

These replacements are necessary because pointwise control does not determine the leading time-integrated term. For instance, if $G(s)=1$ and $U_s=e^{-2\pi\I s/t}$, then $|U_sG(s)|=|G(s)|$ for every $s\in[0,t]$, but $\int_0^tG(s)\,ds=t$ whereas $\int_0^tU_sG(s)\,ds=0$. Thus pointwise unitarity of $e^{-\I sV_\mu}$, or pointwise boundedness of $g_{t,s,\sigma,\tau}(0)$, is not enough to identify the leading coefficient. The precise replacement statements and their complete proofs are given in
\cref{app:delta-replacement-proofs}. For the same reason, the two non-delta terms in \cref{eq:lower_error_three_term_decomposition} must also be controlled
after the full time integration in $L^2$; these estimates are given in \cref{sec:lower-bound-nondelta}, with the zeroth-order proof deferred to \cref{app:non_delta}.
We therefore obtain the following theorem:
\begin{thm}[Lower bound for the exact delta contribution]
\label[thm]{thm:exact_delta_contribution_t54}
Assume that $H_\mu\psi_\ast=E_0\psi_\ast$ and $\psi_\ast(0)\neq0$. In
\cref{eq:exact_delta_contribution_sigma}, take $\psi_0=\psi_\ast$. Then, for every fixed $\sigma\ge0$, there exist constants
$c_{\delta,\psi_\ast}>0$ and $t_{\delta,\mu,\psi_\ast,\sigma}>0$ such that
\begin{equation}
\label{eq:exact_delta_contribution_t54_lower}
\left\|
-4\pi
\int_0^t\int_0^s
e^{-\I sV_\mu}
e^{\I\tau\Delta}
\left(
g_{t,s,\sigma,\tau}(0)\delta_0
\right)
\,d\tau\,ds
\right\|_{L^2(\RR^3)}
\ge
c_{\delta,\psi_\ast}t^{5/4},
\quad
0<t<t_{\delta,\mu,\psi_\ast,\sigma}.
\end{equation}
\end{thm}
\begin{proof}
By the delta replacement result proved in \cref{app:delta-replacement-proofs}, the exact delta contribution satisfies
\begin{equation}
\label{eq:delta_replaced_by_free_term_proof}
\begin{aligned}
-4\pi\int_0^t\int_0^s
e^{-\I sV_\mu}
e^{\I\tau\Delta}
\left(
g_{t,s,\sigma,\tau}(0)\delta_0
\right)
\,d\tau\,ds                        =-4\pi e^{-\I E_0\sigma}\psi_\ast(0)F_\Delta(t,\cdot)
+r_{\delta,\sigma}(t),
\end{aligned}
\end{equation}
where $\|r_{\delta,\sigma}(t)\|_{L^2(\RR^3)}
=o(t^{5/4})$ as $t\to0$. Thus, by the reverse triangle inequality and \cref{lem:free_delta_contribution_scaling},
\begin{equation}
\begin{aligned}
&\left\|
-4\pi
\int_0^t\int_0^s
e^{-\I sV_\mu}
e^{\I\tau\Delta}
\left(
g_{t,s,\sigma,\tau}(0)\delta_0
\right)
\,d\tau\,ds
\right\|_{L^2(\RR^3)}\\
&\quad\ge
\left\|
4\pi e^{-\I E_0\sigma}\psi_\ast(0)F_\Delta(t,\cdot)
\right\|_{L^2(\RR^3)}
-\|r_{\delta,\sigma}(t)\|_{L^2(\RR^3)} =
4\pi|\psi_\ast(0)|C_\Delta t^{5/4}
-o(t^{5/4}),
\end{aligned}
\end{equation}
since $|e^{-\I E_0\sigma}|=1$. Because $\psi_\ast(0)\neq0$ and
$C_\Delta>0$, we can choose
$t_{\delta,\mu,\psi_\ast,\sigma}>0$ sufficiently small so that, for
$0<t<t_{\delta,\mu,\psi_\ast,\sigma}$, we have $o(t^{5/4})\le
2\pi|\psi_\ast(0)|C_\Delta t^{5/4}$.
Hence \cref{eq:exact_delta_contribution_t54_lower} follows.
\end{proof}

\subsection{Lower-Order Estimates for the Non-delta Contributions}
\label{sec:lower-bound-nondelta}
We now estimate the second and third summands in
\cref{eq:lower_error_three_term_decomposition}. 
Since $V_\mu$ is real-valued,
the Borel functional calculus implies that $e^{-\I sV_\mu}$ is multiplication
by $e^{-\I sV_\mu(x)}$; see \cite[Theorem VIII.5]{ReedSimon1980}. Hence
$\|e^{-\I sV_\mu}f\|_{L^2(\RR^3)}
=\|f\|_{L^2(\RR^3)},\ 
f\in L^2(\RR^3)$,
so the left Yukawa phase does not affect the $L^2$ upper estimates below.
The second summand
$\mu^2V_\mu g_{t,s,\sigma,\tau}$ is the zeroth-order non-delta contribution.
Its proof is a standard estimate using $\mu^2V_\mu\in L^2(\RR^3)$, free unitarity, and the Sobolev embedding $H^2(\RR^3)\hookrightarrow L^\infty(\RR^3)$; we state the result here and give the proof in \cref{app:non_delta}. The third summand $2\nabla V_\mu\cdot\nabla g_{t,s,\sigma,\tau}$ is the first-order non-delta contribution, which is the more singular and thus leading non-delta term. Its proof uses duality together with Kato smoothing; we include this proof below.

\begin{lemma}[Non-delta estimate of the zeroth-order term ($\Delta v$)]\label[lemma]{lem:exterior_zero_order}
Let $V_\mu(x)$ be the Yukawa potential as in \cref{eq:yukawa_nbody_potential}, and let $t>0$. For $f\in H^2(\mathbb R^3)$, 
there exists a constant $C_\mu>0$, depending only on $\mu$, such that
\begin{equation}
\left\| \int_0^t\int_0^s e^{i\tau\Delta}\mu^2V_\mu e^{i(s-\tau)\Delta}f \,d\tau\,ds \right\|_{L^2(\mathbb R^3)} \le C_\mu t^2\|f\|_{H^2(\mathbb R^3)}.
\end{equation}
\end{lemma}

\begin{lemma}[Non-delta estimate of the first-order term $(\nabla v\cdot\nabla)$]\label[lemma]{lem:exterior_first_order}
Under the notation of \cref{lem:exterior_zero_order}, 
there exists a constant
$\widetilde C_\mu>0$, depending only on $\mu$, such that, for $0<t\le 1$,
\begin{equation}
\left\|
2\int_0^t\int_0^s
e^{i\tau\Delta}
\bigl(\nabla V_\mu\cdot\nabla\bigr)
e^{i(s-\tau)\Delta}f
\,d\tau\,ds
\right\|_{L^2(\mathbb R^3)}
\le
\widetilde C_\mu t^{3/2}
\|f\|_{H^2(\mathbb R^3)} .
\end{equation}
\end{lemma}

\begin{proof}
Using the explicit formula for the Yukawa potential, we have, for $x\neq 0$,
\begin{equation}
\nabla V_\mu(x)
=\frac{e^{-\mu |x|}}{|x|^2}(1+\mu |x|)\frac{x}{|x|}
=\frac{e^{-\mu |x|}}{|x|^2}\frac{x}{|x|}
+\mu\frac{e^{-\mu |x|}}{|x|}\frac{x}{|x|}.
\end{equation}
Hence
$|\nabla V_\mu(x)|
\le e^{-\mu |x|}/|x|^2
+\mu e^{-\mu |x|}/|x|$.
For $f\in H^2(\mathbb R^3)$, set
\begin{equation}
A_t f :=
2\int_0^t\int_0^s
e^{i\tau\Delta}
(\nabla V_\mu\cdot\nabla)
e^{i(s-\tau)\Delta}f
\,d\tau\,ds .
\end{equation}
We split $A_tf=A_t^{(1)}f+A_t^{(2)}f$, where $A_t^{(1)}f$ is the contribution
from $\mu e^{-\mu |x|}|x|^{-1}\nabla$ and $A_t^{(2)}f$ is the contribution
from $e^{-\mu |x|}|x|^{-2}\nabla$. The less singular part $A_t^{(1)}f$ is estimated directly in $L^2$: by Minkowski's inequality and the unitarity of $e^{i\tau\Delta}$ on $L^2$, we have

\begin{equation}\label{eq:At_1}
\begin{aligned}
\|A_t^{(1)}f\|_{L^2}
&\le
2\mu\int_0^t\int_0^s
\left\|
e^{i\tau\Delta}
\left(
\frac{e^{-\mu |x|}}{|x|}
\frac{x}{|x|}\cdot
\nabla e^{i(s-\tau)\Delta}f
\right)
\right\|_{L^2}
\,d\tau\,ds \\
&=
2\mu\int_0^t\int_0^s
\left\|
\frac{e^{-\mu |x|}}{|x|}
\frac{x}{|x|}\cdot
\nabla e^{i(s-\tau)\Delta}f
\right\|_{L^2}
\,d\tau\,ds 
\le
2\mu\int_0^t\int_0^s
\left\|
\frac{1}{|x|}
\nabla e^{i(s-\tau)\Delta}f
\right\|_{L^2}
\,d\tau\,ds .
\end{aligned}
\end{equation}
By \cref{lem:hls_operator}, applied componentwise with $\phi=\partial_{x_j}e^{i(s-\tau)\Delta}f$, we have
\begin{equation}\label{eq:less_sing_hls_p2}
\left\|
\frac{1}{|x|}
\nabla e^{i(s-\tau)\Delta}f
\right\|_{L^2}^2
=\sum_{j=1}^3
\left\|
\frac{1}{|x|}
\partial_{x_j}e^{i(s-\tau)\Delta}f
\right\|_{L^2}^2 
\le 4\sum_{j=1}^3
\left\|
|p|\partial_{x_j}e^{i(s-\tau)\Delta}f
\right\|_{L^2}^2 
=4\left\|
|p|^2e^{i(s-\tau)\Delta}f
\right\|_{L^2}^2 .
\end{equation}
Since $|p|^2$ commutes with $e^{i(s-\tau)\Delta}$, we get
\begin{equation}\label{eq:less_sing_p2_commute}
\left\|
|p|^2e^{i(s-\tau)\Delta}f
\right\|_{L^2(\mathbb R^3)}
=\left\|
|p|^2f
\right\|_{L^2(\mathbb R^3)}
\le
\|f\|_{H^2(\mathbb R^3)}.
\end{equation}
Combining \cref{eq:less_sing_hls_p2,eq:less_sing_p2_commute}, we obtain
\begin{equation}\label{eq:less_sing_grad_bound}
\left\|
\frac{1}{|x|}
\nabla e^{i(s-\tau)\Delta}f
\right\|_{L^2(\mathbb R^3)}
\le C\|f\|_{H^2(\mathbb R^3)}.
\end{equation}
Therefore, substituting \cref{eq:less_sing_grad_bound} into \cref{eq:At_1}, we have
\begin{equation} \label{eq:J1_less_sing}
\|A_t^{(1)}f\|_{L^2}
\le
C_\mu
\int_0^t\int_0^s
\|f\|_{H^2}
\,d\tau\,ds
\le
C_\mu t^2\|f\|_{H^2}.
\end{equation}
It remains to estimate the more singular contribution $A_t^{(2)}f$,
coming from $e^{-\mu |x|}|x|^{-2}\nabla$. This is the term for which
the direct $L^2$ argument above is not sufficient, and we estimate it
by duality and Kato smoothing.
By
\cref{lem:op_norm_dual}, it suffices to estimate the $L^2(\mathbb R^3)$ pairing
\begin{equation}
\langle g, A_t^{(2)}f\rangle_{L^2(\mathbb R^3)},
\quad
g\in L^2(\mathbb R^3),\quad f\in H^2(\mathbb R^3).
\end{equation}
Equivalently, it is enough to prove that
$|\langle g,A_t^{(2)}f\rangle_{L^2(\mathbb R^3)}|
\le
\widetilde C_\mu t^{3/2}
\|g\|_{L^2(\mathbb R^3)}
\|f\|_{H^2(\mathbb R^3)}$
for every $g\in L^2(\mathbb R^3)$. Taking the supremum over
$\|g\|_{L^2}\le 1$ then gives the desired $L^2$ estimate.
Since $(e^{i\tau\Delta})^\ast=e^{-i\tau\Delta}$, we have
$\left\langle
g,e^{i\tau\Delta}q
\right\rangle_{L^2}
=\left\langle
e^{-i\tau\Delta}g,q
\right\rangle_{L^2}$
for every $q\in L^2(\mathbb R^3)$.
For $0\le \tau\le s\le t$, we write
$h_\tau:=e^{-i\tau\Delta}g,
\quad
u_{s,\tau}:=e^{i(s-\tau)\Delta}f$.
Then
\begin{equation}
|\langle g,A_t^{(2)}f\rangle_{L^2}|
\le
2\int_0^t\int_0^s
\int_{\mathbb R^3}
\frac{e^{-\mu |x|}}{|x|^2}
|h_\tau(x)|\,|\nabla u_{s,\tau}(x)|
\,dx\,d\tau\,ds .
\end{equation}
By the pointwise bound on $\nabla v$ and the Cauchy-Schwarz inequality in the spatial variable, 
\begin{equation}\label{eq:I1_bound} \begin{aligned} 
2\int_0^s
\left|
\left\langle
h_\tau,\,
\frac{e^{-\mu |x|}}{|x|^2}
\frac{x}{|x|}\cdot\nabla u_{s,\tau}
\right\rangle_{L^2}
\right|
\,d\tau  
&\le
2\int_0^s
\int_{\mathbb R^3}
\frac{e^{-\mu |x|}}{|x|^2}
|h_\tau(x)|\,|\nabla u_{s,\tau}(x)|
\,dx\,d\tau  \\
&\le
2\int_0^s
\left\|
\frac{e^{-\mu |x|}}{|x|}h_\tau
\right\|_{L^2(\mathbb R^3)}
\left\|
\frac{1}{|x|}\nabla u_{s,\tau}
\right\|_{L^2(\mathbb R^3)}
\,d\tau .
\end{aligned} 
\end{equation}
By the same componentwise HLS estimate used above in \cref{eq:less_sing_grad_bound}, applied to $u_{s,\tau}=e^{i(s-\tau)\Delta}f$, we have
$\||x|^{-1}\nabla u_{s,\tau}\|_{L^2(\RR^3)}
\le C\|f\|_{H^2(\RR^3)}$. Hence \cref{eq:I1_bound} gives
\begin{equation}\label{eq:J1_reduced}
2\int_0^s
\int_{\RR^3}
\frac{e^{-\mu|x|}}{|x|^2}
|h_\tau(x)|\,|\nabla u_{s,\tau}(x)|
\,dx\,d\tau                           \le
\tilde C
\|f\|_{H^2(\RR^3)}
\int_0^s
\left\|
\frac{e^{-\mu|x|}}{|x|}h_\tau
\right\|_{L^2(\RR^3)}
\,d\tau,
\end{equation}
where $h_\tau:=e^{-i\tau\Delta}g$.
To estimate the time integral, note first that $e^{-\mu|x|}\le 1$, so
\begin{equation}
\left\|
\frac{e^{-\mu|x|}}{|x|}e^{-i\tau\Delta}g
\right\|_{L^2(\mathbb R^3)}
\le
\left\|
\frac{1}{|x|}e^{-i\tau\Delta}g
\right\|_{L^2(\mathbb R^3)}.
\end{equation}
Hence, by H\"older's inequality in the time variable 
\begin{equation}
\int_0^s
\left\|
\frac{e^{-\mu|x|}}{|x|}e^{-i\tau\Delta}g
\right\|_{L^2}
\,d\tau
\le
\int_{\mathbb R}
\mathbf \chi_{[0,s]}(\tau)
\left\|
\frac{1}{|x|}e^{-i\tau\Delta}g
\right\|_{L_x^2}
\,d\tau 
\le
\|\mathbf \chi_{[0,s]}(\tau)\|_{L_t^2}
\left\|
\frac{1}{|x|}e^{-i\tau\Delta}g
\right\|_{L_t^2L_x^2}.
\end{equation}
and the Kato smoothing in \cref{lem:kato} applies equally to $e^{-i\tau\Delta}$ by the change of variables $\tau\mapsto-\tau$, so
\begin{equation}\label{eq:kato}
    \left\|
\frac{1}{|x|}e^{-i\tau\Delta}g
\right\|_{L_t^2L_x^2} \leq \tilde{C}\|g\|_{L_x^2}=\tilde{C}\|g\|_{L^2}.
\end{equation}
We conclude that
\begin{equation} \label{eq:I1_kato}
\int_0^s
\left\|
\frac{e^{-\mu|x|}}{|x|}e^{-i\tau\Delta}g
\right\|_{L^2(\mathbb R^3)}
\,d\tau
\le
\tilde{C}s^{1/2}\|g\|_{L^2(\mathbb R^3)}.
\end{equation}
Combining \cref{eq:J1_reduced} and \cref{eq:I1_kato}, we obtain that the contribution of the term $e^{-\mu |x|}|x|^{-2}\nabla$ is bounded by
$\tilde C_\mu s^{1/2}
\|g\|_{L^2(\mathbb R^3)}
\|f\|_{H^2(\mathbb R^3)}$.
Integrating the above bound in $s$, we get
\begin{equation}
|\langle g,A_t^{(2)}f\rangle_{L^2}|
\le
\tilde C_\mu
\int_0^t s^{1/2}\,ds\,
\|g\|_{L^2(\mathbb R^3)}
\|f\|_{H^2(\mathbb R^3)} 
\le
\tilde C_\mu t^{3/2}
\|g\|_{L^2(\mathbb R^3)}
\|f\|_{H^2(\mathbb R^3)}.
\end{equation}
Taking the supremum over all $g\in L^2(\mathbb R^3)$ with
$\|g\|_{L^2}\le 1$ gives
$\|A_t^{(2)}f\|_{L^2}
\le
\tilde C_\mu t^{3/2}\|f\|_{H^2}$.
Combining this estimate with \cref{eq:J1_less_sing}, we obtain
\begin{equation}
\|A_tf\|_{L^2}
\le
C_\mu t^2\|f\|_{H^2}+\tilde C_\mu t^{3/2}\|f\|_{H^2}
\le
\widetilde C_\mu t^{3/2}\|f\|_{H^2},
\quad 0<t\le1.
\end{equation}

\end{proof}

\subsection{Proof of \cref{thm:main-one-body-yukawa-lower}} \label{subsubsec: pf main lower}
Recall the one-step error $E_\mu(t):=U_{1,\mu}(t)-e^{-\I tH_\mu}$.
Since $\psi_\ast$ is a ground state of $H_\mu$, there exists a ground-state
energy $E_0\in\RR$ such that $H_\mu\psi_\ast=E_0\psi_\ast$.
Taking $\sigma=0$ and $\psi_0=\psi_\ast$ in \cref{eq:lower_error_three_term_decomposition}, the error $E_\mu(t)\psi_\ast$ is the sum of the delta contribution, the zeroth-order non-delta contribution, and the first-order non-delta contribution.
By \cref{thm:exact_delta_contribution_t54}, the delta contribution is bounded from below by $c_{\delta,\psi_\ast}t^{5/4}$ for all sufficiently small $t>0$.
By \cref{lem:exterior_zero_order,lem:exterior_first_order}, the two non-delta contributions are bounded from above by $C_{\mu,\psi_\ast}(t^2+t^{3/2})$. Hence, by the reverse triangle inequality,
$\|E_\mu(t)\psi_\ast\|_{L^2(\RR^3)}
\ge
c_{\delta,\psi_\ast}t^{5/4}
-C_{\mu,\psi_\ast}(t^2+t^{3/2})$
for all sufficiently small $t>0$.
Since $t^2+t^{3/2}=o(t^{5/4})$ as $t\to0$, we may choose
$t_{0,\mu,\psi_\ast}>0$ such that, for all
$0<t<t_{0,\mu,\psi_\ast}$, we have $C_{\mu,\psi_\ast}(t^2+t^{3/2})
\le
\frac12 c_{\delta,\psi_\ast}t^{5/4}$.
Therefore,
$\|E_\mu(t)\psi_\ast\|_{L^2(\RR^3)}
\ge
\frac12 c_{\delta,\psi_\ast}t^{5/4}$.
Since $E_\mu(t)\psi_\ast
=U_{1,\mu}(t)\psi_\ast
-e^{-\I tH_\mu}\psi_\ast$,
we obtain
\begin{equation}
\left\|
U_{1,\mu}(t)\psi_\ast
-e^{-\I tH_\mu}\psi_\ast
\right\|_{L^2(\RR^3)}
\ge
c_{\psi_\ast}t^{5/4},
\quad
0<t<t_{0,\mu,\psi_\ast},
\end{equation}
with $c_{\psi_\ast}:=\frac12c_{\delta,\psi_\ast}>0$. This proves
\cref{thm:main-one-body-yukawa-lower}.

\section{Numerical method}\label{sect: numerics}
\subsection{Spherical harmonics reduction and Galerkin discretization for the Yukawa Hamiltonian}\label{subsec:yukawa_galerkin}
In this subsection we describe the numerical discretization used for the Yukawa Hamiltonian. Since the Yukawa potential is radial, the angular and radial variables separate naturally in spherical coordinates. We first expand the solution in the spherical harmonic basis on $S^2$, then reduce the dynamics to the radial $\ell=0$ sector for radial initial data, and finally apply an exact Galerkin projection in the corresponding one-dimensional radial variable.

We consider the time-dependent Schr\"odinger equation
\begin{equation}\label{eq:yukawa_tdse_full}
    i\partial_t \psi(x,t)
    =H_Y\psi(x,t),
    \quad
    H_Y
    :=-\Delta-Z\frac{e^{-\mu |x|}}{|x|},
    \quad x\in\mathbb{R}^3,
\end{equation}
where $Z>0$ and $\mu>0$. Writing $x=r\omega$ with
    $r:=|x|\in(0,\infty)$, and
    $\omega:=x/|x|\in S^2$,
separates the spatial variable into its radial and angular parts. Since the Yukawa potential depends only on $r$, the angular dependence is naturally expanded in the spherical harmonic basis on $S^2$. The spherical harmonics $Y_\ell^m(\omega)$ are the standard orthonormal eigenfunctions of the angular Laplacian $\Delta_{S^2}$ on the unit sphere. Since they are defined on $S^2$, they depend only on the angular variable $\omega$.
In these coordinates, the Laplacian takes the standard form
\begin{equation}\label{eq:laplacian_spherical_yukawa}
    \Delta=\partial_{rr}+\frac{2}{r}\partial_r+\frac{1}{r^2}\Delta_{S^2},
\end{equation}
where $\Delta_{S^2}$ denotes the Laplace-Beltrami operator on the unit sphere.

Let
\begin{equation}
\{Y_\ell^m:\ \ell=0,1,2,\dots,\ -\ell\le m\le \ell\}
\end{equation}
be the standard spherical harmonics, which form an orthonormal family in $L^2(S^2)$:
\begin{equation}\label{eq:sh_orthogonality_yukawa}
    \langle Y_\ell^m, Y_{\ell'}^{m'}\rangle_{L^2(S^2)}   =\int_{S^2}Y_\ell^m(\omega)\,\overline{Y_{\ell'}^{m'}(\omega)}
    \,d\omega
=\delta_{\ell\ell'}\delta_{mm'}.
\end{equation}
They satisfy the eigenvalue equation
\begin{equation}\label{eq:sh_eigenvalue_yukawa}
    -\Delta_{S^2}Y_\ell^m
    =\ell(\ell+1)Y_\ell^m.
\end{equation}
We therefore expand the solution as
\begin{equation}\label{eq:spherical_harmonic_expansion_yukawa}
    \psi(r,\omega,t)
=\sum_{\ell=0}^{\infty}\sum_{m=-\ell}^{\ell}u_{\ell m}(r,t)Y_\ell^m(\omega).
\end{equation}
Substituting \cref{eq:spherical_harmonic_expansion_yukawa} into \cref{eq:yukawa_tdse_full}, using \cref{eq:laplacian_spherical_yukawa} and \cref{eq:sh_eigenvalue_yukawa}, and then multiplying by $\overline{Y_{\ell'}^{m'}}$ and integrating over $S^2$, the orthonormality relation \eqref{eq:sh_orthogonality_yukawa} isolates the $(\ell',m')$ component. Thus the angular modes decouple, and each coefficient $u_{\ell m}$ satisfies
\begin{equation}\label{eq:ulm_radial_equation_yukawa}
    i\partial_t u_{\ell m}(r,t)
    =-\left(\partial_{rr}u_{\ell m}(r,t)+\frac{2}{r}\partial_r u_{\ell m}(r,t)-\frac{\ell(\ell+1)}{r^2}u_{\ell m}(r,t)
    \right)-Z\frac{e^{-\mu r}}{r}u_{\ell m}(r,t).
\end{equation}
Thus the radial Yukawa potential does not mix different spherical harmonic modes. In the numerical experiment, we choose radial initial data $\psi(x,0)=\psi_0(|x|)$. Equivalently, in the expansion \cref{eq:spherical_harmonic_expansion_yukawa}, only the mode $(\ell,m)=(0,0)$ is nonzero initially. By the decoupling in \cref{eq:ulm_radial_equation_yukawa}, all higher angular modes remain zero for all times, so the radial sector is invariant under the Yukawa flow. Therefore, it suffices to study the coefficient $u_{00}(r,t)$, which solves
\begin{equation}\label{eq:u00_equation_yukawa}
    i\partial_t u_{00}(r,t)
    =-\left(\partial_{rr}u_{00}(r,t)+\frac{2}{r}\partial_r u_{00}(r,t)\right)-Z\frac{e^{-\mu r}}{r}u_{00}(r,t).
\end{equation}
To remove the first-order radial derivative, we introduce
\begin{equation}\label{eq:w_definition_yukawa}
    w(r,t):=ru_{00}(r,t).
\end{equation}
Then $u_{00}(r,t)=r^{-1}w(r,t)$, and a direct computation gives
\begin{equation}\label{eq:radial_identity_yukawa}
    \partial_{rr}\!\left(\frac{w}{r}\right)+\frac{2}{r}\partial_r\!\left(\frac{w}{r}\right)
    =\frac{1}{r}w''.
\end{equation}
Multiplying \cref{eq:u00_equation_yukawa} by $r$ and using \cref{eq:radial_identity_yukawa}, we obtain the one-dimensional radial equation
\begin{equation}\label{eq:w_radial_pde_yukawa}
    i\partial_t w(r,t)
    = -\,w''(r,t) - Z\frac{e^{-\mu r}}{r}\,w(r,t),
    \quad r>0.
\end{equation}
For the numerical approximation we truncate the radial domain to $(0,R)$ and impose homogeneous Dirichlet boundary conditions
$w(0,t)=w(R,t)=0$.
Accordingly, we consider the radial Yukawa Hamiltonian on $(0,R)$ with homogeneous Dirichlet boundary conditions,
\begin{equation}\label{eq:truncated_radial_operator_yukawa}
H_R:=-\frac{d^2}{dr^2}-Z\frac{e^{-\mu r}}{r}.
\end{equation}
We next discretize \cref{eq:w_radial_pde_yukawa} by a Galerkin projection. 
On the truncated interval $(0,R)$, after the transformation \cref{eq:w_definition_yukawa}, the $\ell=0$ radial problem is naturally associated with the one-dimensional Dirichlet Laplacian $-\frac{d^2}{dr^2}$. 
We therefore use the standard orthonormal eigenfunction basis
\begin{equation}\label{eq:sine_basis_yukawa}
    \chi_n(r)
    :=\sqrt{\frac{2}{R}}
    \sin\!\left(\frac{n\pi r}{R}\right),
    \qquad n\in\mathbb{N},
\end{equation}
which satisfies
\begin{equation}\label{eq:sine_basis_orthogonality_yukawa}
    \int_0^R \chi_m(r)\chi_n(r)\,dr = \delta_{mn}, \quad \mbox{and} \quad
    -\chi_n''(r)=\left(\frac{n\pi}{R}\right)^2\chi_n(r),
    \quad
    \chi_n(0)=\chi_n(R)=0.
\end{equation}
Equivalently, this basis corresponds to the $\ell=0$ radial spherical Bessel modes after the transformation \cref{eq:w_definition_yukawa}. Indeed, the radial part of the separated problem is naturally expressed in terms of spherical Bessel functions, and for $\ell=0$ the regular radial mode is proportional to
$j_0(z)=\sin z/z$.
Hence, after the transformation \cref{eq:w_definition_yukawa}, these modes become proportional to sine functions on $(0,R)$.
For a fixed truncation parameter $K\in\mathbb{N}$, we approximate the solution in $\operatorname{span}\{\chi_1,\dots,\chi_K\}$
and seek an approximate solution of the form
\begin{equation}\label{eq:wK_ansatz_yukawa}
    w_K(r,t)=\sum_{n=1}^K c_n(t)\chi_n(r).
\end{equation}
Using the ansatz \cref{eq:wK_ansatz_yukawa}, we have
\begin{equation}
\partial_t w_K
=\sum_{n=1}^K \dot c_n(t)\chi_n,
\quad
w_K''
=\sum_{n=1}^K c_n(t)\chi_n''.
\end{equation}
Substituting this expansion into \cref{eq:w_radial_pde_yukawa}, taking the $L^2(0,R)$ inner product with $\chi_m$, and using
\begin{equation}
\langle \chi_m,\chi_n\rangle_{L^2(0,R)}=\delta_{mn},
\quad
-\chi_n''=\left(\frac{n\pi}{R}\right)^2\chi_n,
\end{equation}
we obtain
\begin{align}
    i\dot c_m(t)
    &=
    \left\langle
        \chi_m,
        - w_K''(\cdot,t)
    \right\rangle_{L^2(0,R)}
    +
    \left\langle
        \chi_m,
        -Z\frac{e^{-\mu r}}{r}w_K(\cdot,t)
    \right\rangle_{L^2(0,R)}
    \notag\\
    &=
    \sum_{n=1}^K
    \left[
        \left(\frac{n\pi}{R}\right)^2\delta_{mn}
        -
        \frac{2Z}{R}
        \int_0^R
        \frac{e^{-\mu r}}{r}
        \sin\left(\frac{m\pi r}{R}\right)
        \sin\left(\frac{n\pi r}{R}\right)\,dr
    \right]
    c_n(t).
    \label{eq:cm_equation_yukawa}
\end{align}
Writing
    $c(t):=(c_1(t),\dots,c_K(t))^\top\in\mathbb{C}^K$,
we obtain the semidiscrete system
    $i\dot c(t)=(A_K+B_K)c(t)$,
where
\begin{equation}\label{eq:AK_yukawa}
    (A_K)_{mn}
    =\left(\frac{n\pi}{R}\right)^2\delta_{mn},
\end{equation}
and
\begin{equation}\label{eq:BK_integral_yukawa}
    (B_K)_{mn}
    =-\frac{2Z}{R}
    \int_0^R
    \frac{e^{-\mu r}}{r}
    \sin\left(\frac{m\pi r}{R}\right)
    \sin\left(\frac{n\pi r}{R}\right)\,dr.
\end{equation}
This system is semidiscrete in the sense that the radial variable has been discretized by the Galerkin projection, whereas time remains continuous and will be discretized subsequently by the Trotter scheme. The matrix $A_K+B_K$ is the Galerkin representation of the radial Yukawa Hamiltonian \cref{eq:truncated_radial_operator_yukawa} on $(0,R)$.
Using the trigonometric identity
    $2\sin\alpha\sin\beta
    =\cos(\alpha-\beta)-\cos(\alpha+\beta)$,
one may rewrite \cref{eq:BK_integral_yukawa} as
\begin{equation}\label{eq:BK_preF_yukawa}
    (B_K)_{mn}
    =-\frac{Z}{R}
    \int_0^R
    \frac{e^{-\mu r}}{r}
    \left[
        \cos\left(\frac{(m-n)\pi r}{R}\right)
        -\cos\left(\frac{(m+n)\pi r}{R}\right)
    \right]dr.
\end{equation}
For $p=0,1,2,\dots$, define
\begin{equation}\label{eq:Fp_definition_yukawa}
    F_p
    :=\int_0^R
    \frac{e^{-\mu r}}{r}
    \left(
        1-\cos\!\left(\frac{p\pi r}{R}\right)
    \right)\,dr.
\end{equation}
Then $F_0=0$, and \cref{eq:BK_preF_yukawa} becomes
\begin{equation}\label{eq:BK_F_formula_yukawa}
    (B_K)_{mn}
    =
    -\frac{Z}{R}\bigl(F_{m+n}-F_{|m-n|}\bigr).
\end{equation}
This is the representation used in the implementation.
Setting
    $\kappa_p:=p\pi/R$,
for $p\geq 1$ one may evaluate \cref{eq:Fp_definition_yukawa} explicitly as
\begin{equation}\label{eq:Fp_closed_form_yukawa}
    F_p
    =\Re\left[
        \log\left(\frac{\mu-i\kappa_p}{\mu}\right)
        +
        \operatorname{Ei}(-\mu R)
        -
        \operatorname{Ei}\!\bigl(-(\mu-i\kappa_p)R\bigr)
    \right],
\end{equation}
where $\operatorname{Ei}$ denotes the exponential integral.

Let
   $ H_K:=A_K+B_K$.
For a given initial coefficient vector $c^0\in\mathbb{C}^K$ and time step size $s=T/N$, the exact semidiscrete evolution and the first-order splitting approximation at time $T>0$ are
\begin{equation}
    c_{\mathrm{ex}}(T)
    =e^{-iT H_K}c^0, \quad
    c_{\mathrm{TS}}(T)
    :=\bigl(e^{-isB_K}e^{-isA_K}\bigr)^N c^0.
    \label{eq:lie_trotter_yukawa}
\end{equation}
We measure the corresponding numerical error by
\begin{equation}\label{eq:discrete_error_yukawa}
    \xi_N^{(K)}(T;c^0)
    :=\bigl\|
        c_{\mathrm{TS}}(T)-c_{\mathrm{ex}}(T)
    \bigr\|_{\ell^2}.
\end{equation}
This is the quantity plotted in the numerical experiments below.

\subsection{Numerical Results}
\label{subsec:numerical_results}
In this subsection, we present numerical experiments for the first-order splitting error of the Yukawa Hamiltonian. We first study the global error under increasing radial Galerkin resolution, then examine the local one-step error for
Coulomb and Yukawa potentials, and finally isolate the dependence on the Yukawa screening parameter $\mu$.

All experiments use the radial $\ell=0$ Galerkin discretization derived in \cref{subsec:yukawa_galerkin}. Thus the radial problem is reduced to the
one-dimensional equation on $(0,R)$ by the transformation $w(r,t)=ru_{00}(r,t)$ and then discretized in the sine basis. The Galerkin matrices are those in
\cref{eq:AK_yukawa,eq:BK_F_formula_yukawa,eq:Fp_closed_form_yukawa}, and the global error is measured by $\xi_L^{(K)}(T;c^0)$ in \cref{eq:discrete_error_yukawa}. Here and below, $L$ in the plots denotes the number of splitting steps, not the particle number.
We first consider the global-in-time experiment at final time $T=1$, with $Z=1$, $\mu=0.2$, and $R=50$. The initial condition is a common approximate radial Yukawa ground state: it is computed at the largest reference truncation and then projected and renormalized in each smaller Galerkin space. Thus the curves represent the same underlying physical state viewed at different radial resolutions, so their differences mainly reflect the effect of increasing the number of modes.

\begin{figure}[!htb]
    \centering
    \subfloat[Yukawa potential]{
        \includegraphics[width=.47\textwidth]{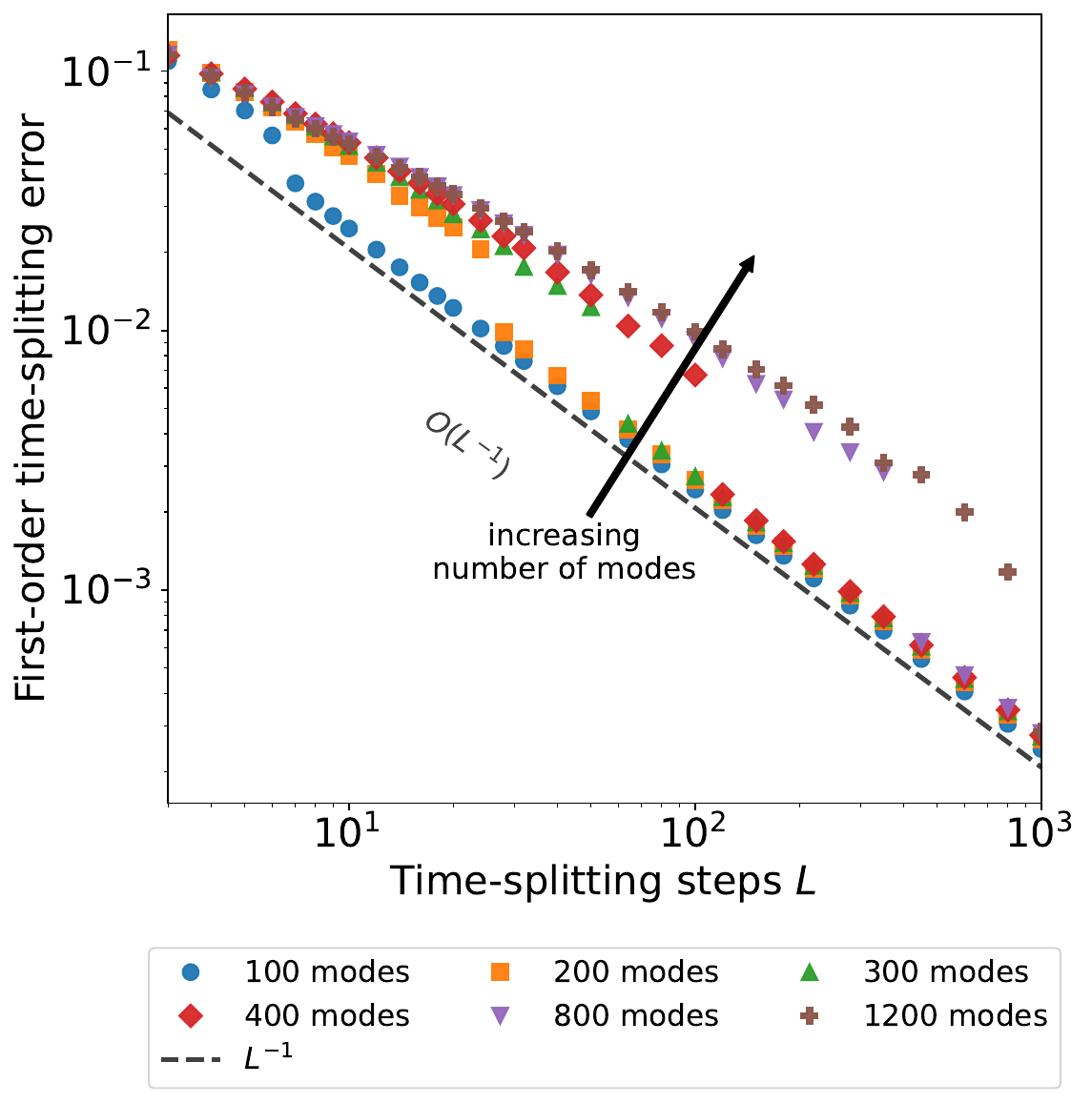}
        \label{fig:fig1}
    }
    \hfill
    \subfloat[Coulomb potential]{
        \includegraphics[width=.47\textwidth]{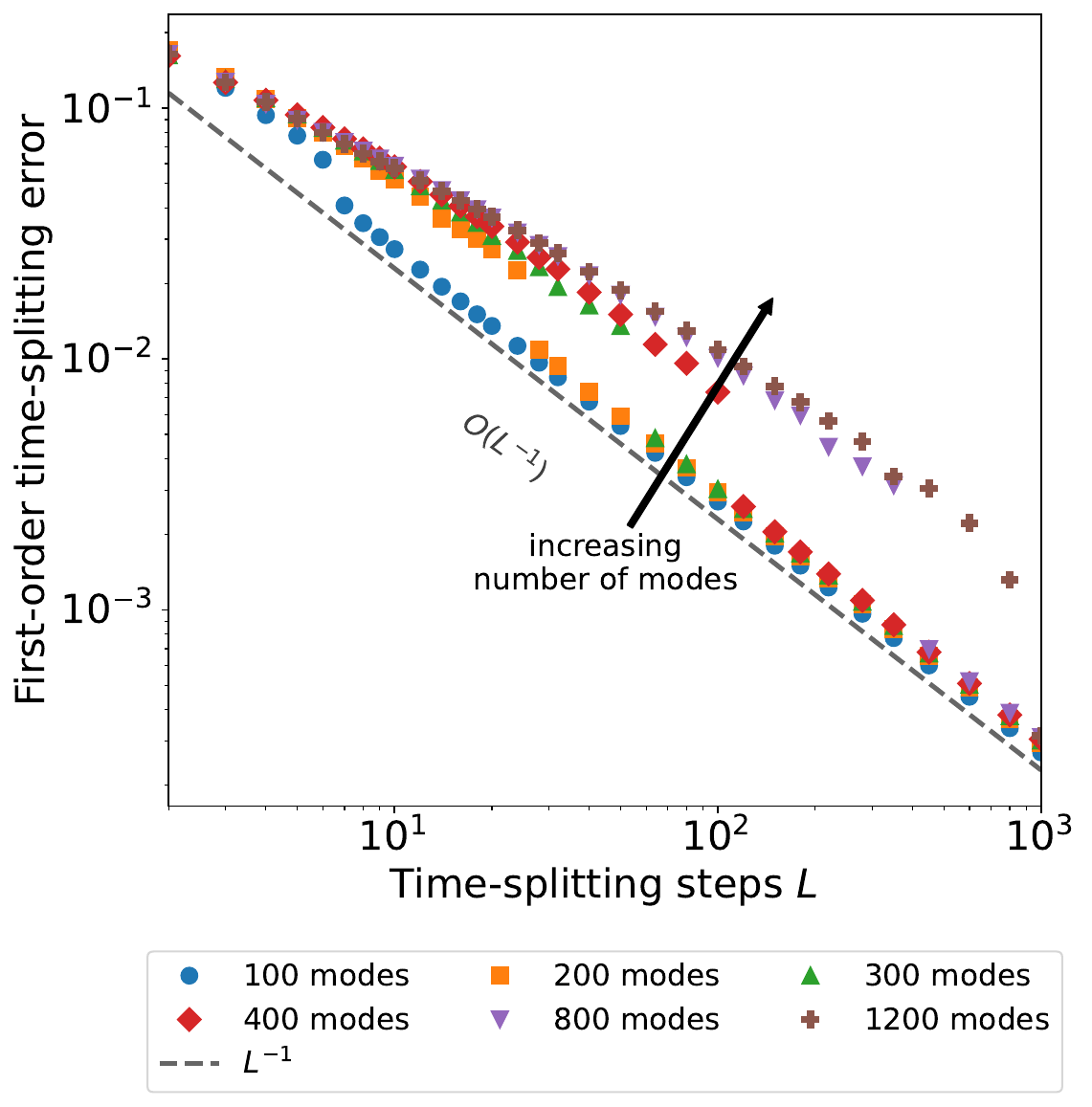}
        \label{fig:fig1_2}
    }
    \caption{(a) Log-log plot of the global first-order splitting error
$\xi_L^{(K)}(T;c^0)=\|c_{\mathrm{TS}}(T)-c_{\mathrm{ex}}(T)\|_{\ell^2}$
versus the number of splitting steps $L$ for the radial Yukawa Hamiltonian at $T=1$, with $Z=1$, $\mu=0.2$, $R=50$, and
$K\in\{100,200,300,400,800,1200\}$. The initial state is a common approximate Yukawa ground state obtained from the largest reference truncation and projected onto each smaller Galerkin space. The dashed line has slope $L^{-1}$ and is included only as a finite-dimensional reference. As $K$ increases, the curves move upward and the apparent transition toward the fixed $K$ first-order regime is delayed. This panel is not intended as direct evidence of a global $L^{-1/4}$ asymptotic.
(b) The same experiment for the radial Coulomb Hamiltonian, with the same $\ell=0$ radial Galerkin discretization, $Z=1$, $R=50$, and $K\in\{100,200,300,400,800,1200\}$. For comparison, we use the same projected initial vector as in panel (a). The Coulomb curves show the same qualitative delayed transition as in the Yukawa case, with slightly larger errors for higher-mode truncations.}
\label{fig:global-singular-comparison}
\end{figure}

The numerical presentation of \cref{fig:global-singular-comparison} closely follows that of \cite{BurgarthFacchiHahnJohnssonYuasa2024}. 
Specifically, we repeat the Coulomb test from \cite{BurgarthFacchiHahnJohnssonYuasa2024} and perform an analogous test for the Yukawa potential for comparison.
For each fixed
truncation level $K$, the semidiscrete problem is finite dimensional, so the
standard $L^{-1}$ regime eventually appears. As $K$ increases, the curves move upward and the onset of this finite-dimensional regime is delayed, indicating that higher radial modes reveal more of the singular infinite-dimensional behavior. The Coulomb panel shows the same qualitative trend as the Yukawa panel, with slightly larger errors for higher truncations. Since the initial state is an approximate ground state, these global curves are not intended to saturate the worst-case $H^2$ rate; the one-step singular mechanism is tested more directly below.
\begin{figure}[!htb]
 \centering
    \centering
    \subfloat[Regularized Gaussian Hamiltonian]{
\includegraphics[height=.34\textheight]{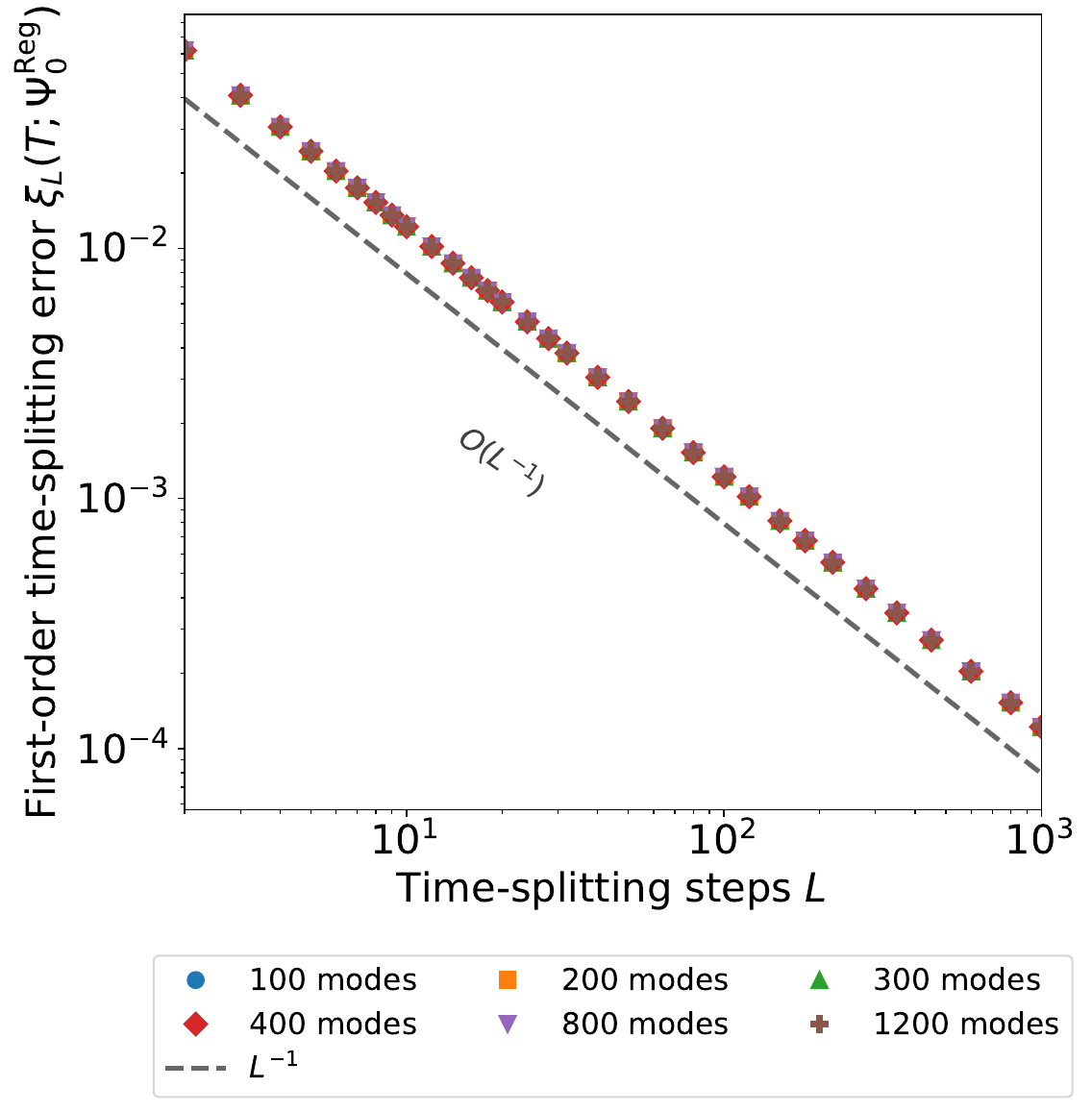}
\label{fig:regular_gaussian_trotter}
    }
     \hfill
    \subfloat[Coulomb and Yukawa local error]{ \includegraphics[height=.34\textheight]{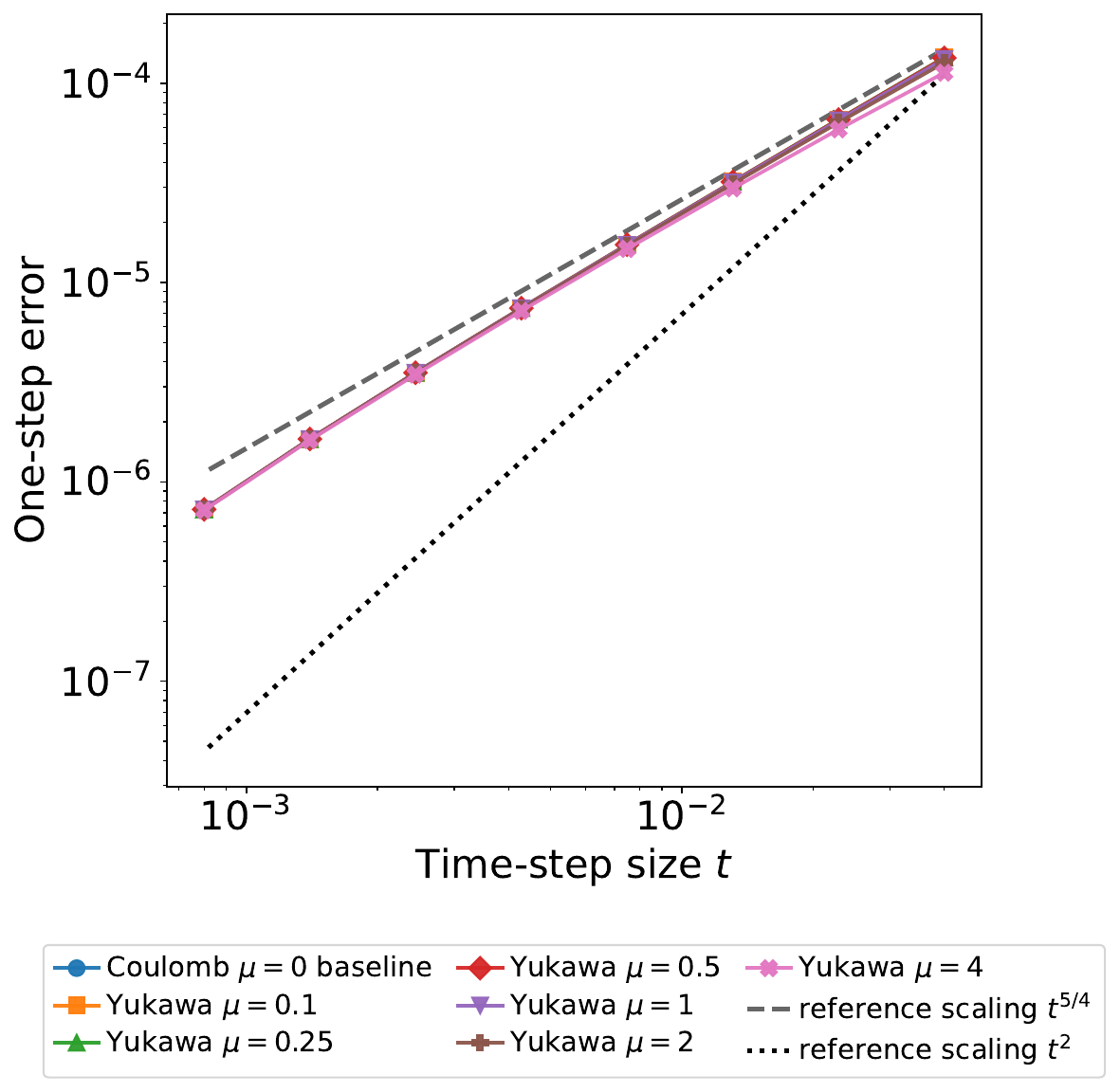}
        \label{fig:fig2}
    }
    \caption{(a) Log-log plot of the global first-order splitting error for the regularized Gaussian Hamiltonian
$V_{\mathrm{Reg}}(r)=-Ze^{-\alpha r^2}$ under the same $\ell=0$ radial Galerkin discretization, with $Z=1$, $\alpha=0.2$, $R=50$, and
$K\in\{100,200,300,400,800,1200\}$. The initial state is a common approximate regularized Gaussian ground state obtained from a reference truncation $K_{\mathrm{ref}}=1200$ and projected onto each smaller Galerkin space. The curves nearly overlap and follow the dashed $L^{-1}$ reference line, indicating standard first-order splitting behavior for this smooth regular potential.
(b) Log-log plot of the one-step local error $\|E_\mu(\Delta t)\|_{\ell^2}$ versus $\Delta t$ for Coulomb $(\mu=0)$ and Yukawa
$\mu\in\{0.1,0.25,0.5,1,2,4\}$, computed with $Z=1$, $R=50$, $K=1200$, and the rough initial state
$c_n^0=C_\alpha n^{-(5/2+\varepsilon)}e^{i\theta_n}$ with $\varepsilon=0.03$ and $\|c^0\|_{\ell^2}=1$. The Coulomb and Yukawa curves nearly overlap and follow the $\Delta t^{5/4}$ reference slope, indicating that the leading local error is governed by the common $1/r$ singularity. The $\Delta t^2$ line is included only for comparison.}
\label{fig:regular-and-rough-local}
\end{figure}

The regularized Gaussian comparison in \cref{fig:regular_gaussian_trotter} separates singular effects from
discretization effects. For the smooth potential $V_{\mathrm{Reg}}(r)=-Ze^{-\alpha r^2}$, the curves for different truncation levels nearly overlap and follow the standard $L^{-1}$ reference line,
consistent with the behavior of regular potentials~\cite{AnFangLin2021}.
Thus the order reduction observed for Yukawa and Coulomb is not a generic artifact of the radial Galerkin discretization, but is tied to the Coulomb-type singularity at the origin and the associated regularity of the initial state.
The Coulomb comparison is useful because Coulomb and Yukawa share the same $1/r$ singularity at the origin, while Yukawa removes the long-range tail. The analysis and the plots are consistent with the same worst-case $L^{-1/4}$ singular mechanism despite exponential screening, indicating that the dominant order reduction is caused by the local singularity rather than by the long-range tail alone. This does not contradict improved convergence for special initial
states: as in the Coulomb case, additional regularity or cancellation can lead to better observed rates, and similar regularity-driven improvement may occur for suitable Yukawa states; see, for example, \cite[Fig.~2]{BurgarthFacchiHahnJohnssonYuasa2024} and \cite{FangWu2026}.

We next study the local one-step error $E_\mu(\Delta t):=e^{-i\Delta tH_\mu}c^0
-e^{-i\Delta tB_\mu}e^{-i\Delta tA}c^0$, where $H_\mu=A+B_\mu$ is the Galerkin Hamiltonian. To reveal the singular local mechanism, we use the rough initial state
\begin{equation}
    c_n^0
    =C_{\alpha}
    n^{-\alpha}e^{i\theta_n},
    \quad
    \alpha=\frac52+\varepsilon,
    \quad
    \varepsilon=0.03,
\label{eq:rough_initial_state_numerics}
\end{equation}
with fixed phases $\theta_n$ and normalization $\|c^0\|_{\ell^2}=1$. This data is just above the $H^2$ threshold, since the sine basis $H^2$ norm contains weights comparable to $n^4$ and hence requires $\alpha>5/2$.
\begin{figure}[!htb]
    \centering
    \subfloat[One-step local Lie-Trotter splitting error]{
\includegraphics[width=.53\textwidth]{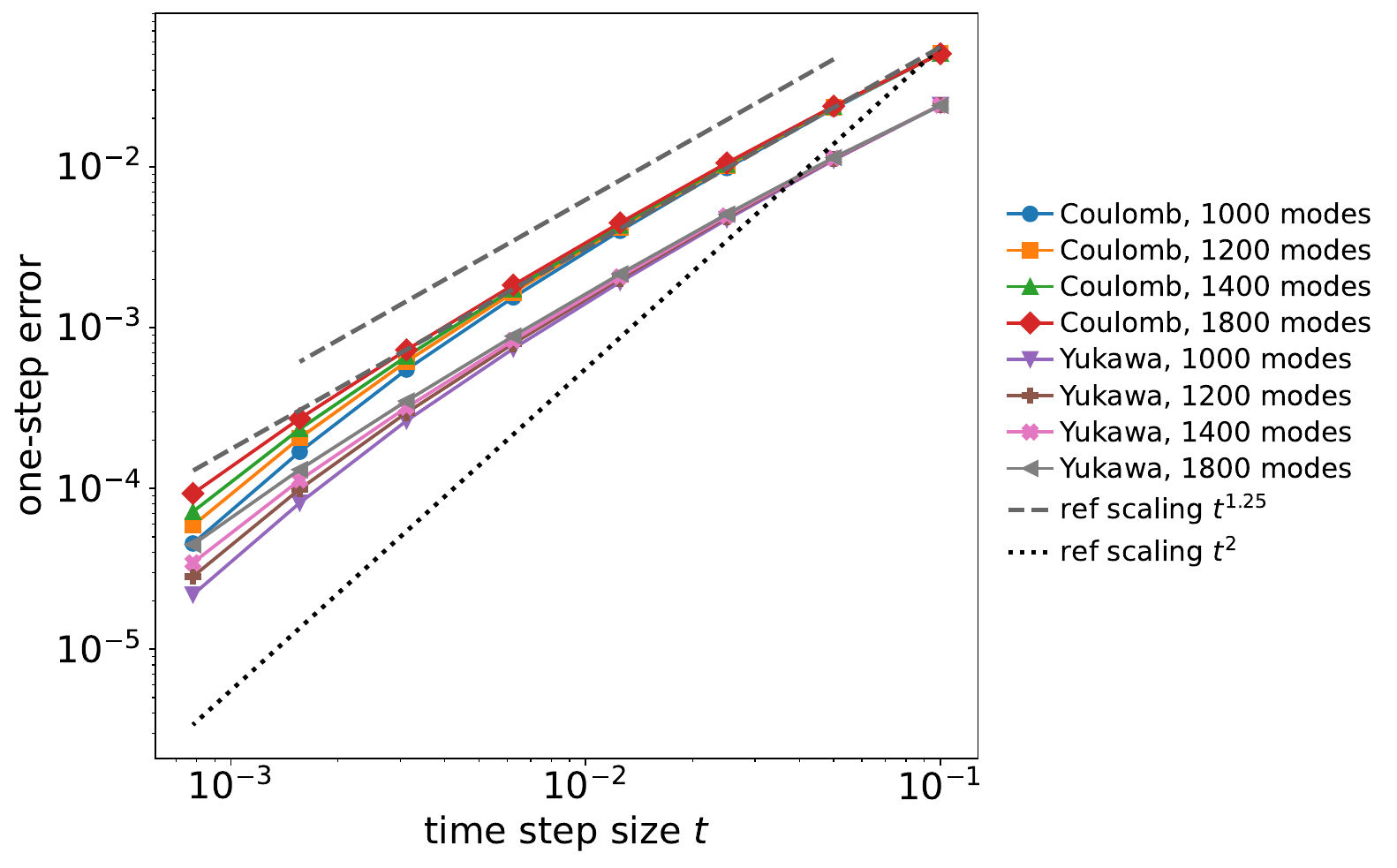}
        \label{fig:fig2_2}
    }
    \hspace{0.02\textwidth}
    \subfloat[Relative local-error norm]{
        \includegraphics[width=.38\textwidth]{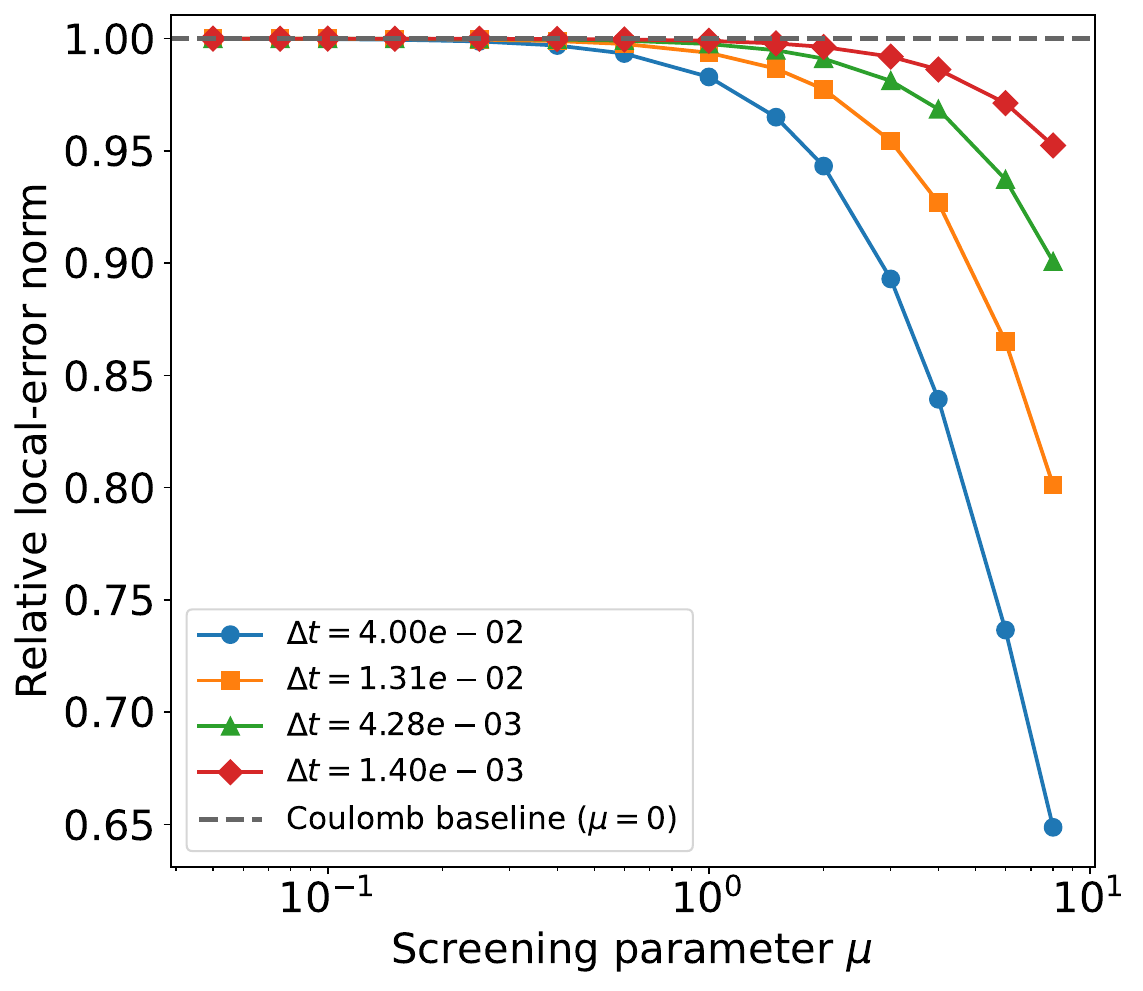}
        \label{fig:fig3_1}
    }
    \caption{(a) Log-log plot of the one-step local splitting error versus the time-step size $t$ for Coulomb and Yukawa potentials, using the corresponding normalized Galerkin ground state as the initial state for each potential. The radial Galerkin truncation uses $R=40$, $Z=1$, Yukawa screening parameter $\mu=1$, and $K\in\{1000,1200,1400,1800\}$. The reference slopes are $t^{5/4}$ and $t^2$. Both Coulomb and Yukawa follow the $t^{5/4}$ scaling; since the initial state is recomputed for each potential, the vertical separation should be interpreted qualitatively rather than as a fixed state screening prefactor.
(b) Semilog plot of the relative local-error norm
$R_\mu(\Delta t)=\|E_\mu(\Delta t)\|_{\ell^2}/\|e_0(\Delta t)\|_{\ell^2}$
as a function of the screening parameter $\mu$, using the fixed rough initial state from \cref{eq:rough_initial_state_numerics}. The dashed horizontal line is the Coulomb baseline $R_\mu=1$. Stronger screening reduces the local error prefactor, while smaller time steps remain closer to the Coulomb baseline because the leading singular contribution is Coulomb-like.}
\label{fig:ground-and-relative-local}
\end{figure}
The rough state local experiment in \cref{fig:fig2} shows that the Coulomb and
Yukawa curves nearly overlap and follow the $\Delta t^{5/4}$ reference slope.
This supports the lower bound mechanism: the leading local error is governed by the common near-origin singularity, since $e^{-\mu r}/r= 1/r-\mu+O(r)$ as $r\to0$.
Thus screening changes the regular part and the tail, but not the leading $1/r$ singularity. The local $\Delta t^{5/4}$ scaling is the one-step mechanism behind the global $1/4$ exponent through the power counting $L(\Delta t)^{5/4}=T(\Delta t)^{1/4}$.
The complementary ground state local experiment in \cref{fig:fig2_2} gives the
same qualitative message. For each truncation level $K$, both Coulomb and Yukawa ground state data remain aligned with the $t^{5/4}$ reference slope rather than the smooth $t^2$ slope. Since the ground state is recomputed for
each potential, vertical differences between the curves should be interpreted qualitatively rather than as a fixed state screening prefactor.

We finally isolate the effect of the screening parameter $\mu$ at fixed rough initial data. Let $e_0(\Delta t)$ denote the Coulomb local error and define
$R_\mu(\Delta t):=
\|E_\mu(\Delta t)\|_{\ell^2}/\|e_0(\Delta t)\|_{\ell^2}$.
Thus $R_\mu(\Delta t)=1$ is the Coulomb baseline. As shown in
\cref{fig:fig3_1}, the ratio stays close to $1$ for small $\mu$ and decreases for larger $\mu$, showing that stronger screening reduces the local error prefactor. For smaller time steps, the ratio remains closer to $1$, consistent with the leading local error being dominated by the common Coulomb-type singular contribution, while the $\mu$-dependent part enters as a more regular correction.

Together,
\cref{fig:global-singular-comparison,fig:regular-and-rough-local,fig:ground-and-relative-local}
support the theoretical picture: Yukawa screening changes constants,
prefactors, and lower-order corrections, as seen most clearly in the $\mu$-dependent diagnostics, but it does not change the leading Coulomb-type singular splitting behavior generated by the unchanged $1/r$ singularity at the origin.

\section{Conclusion and Remarks}\label{sect: conclusion}

In this work, we studied the first-order splitting method for Schr\"odinger
equations with Yukawa interactions. The Yukawa potential is a short-range
screened version of the Coulomb potential: it retains the Coulomb-type singularity at the origin while removing the long-range tail through exponential decay at spatial infinity. We proved that, in the many-body setting, the global splitting error has order $1/4$ in the time step for all initial data in $H^2(\RR^{3N})$, with explicit polynomial dependence on the particle number. This gives a system-size-explicit long-time estimate without regularizing the singular interaction.
We also proved a matching one-body local lower bound. More precisely, the one-step error has a sharp obstruction of order $t^{5/4}$, generated by the unchanged Coulomb-type singularity at the origin. The lower bound analysis
isolates the singular delta contribution coming from the distributional Laplacian of the Yukawa potential and shows that the remaining non-delta and
replacement terms are lower order. Thus, although Yukawa screening improves the behavior of the potential at spatial infinity and changes constants, prefactors, and lower-order corrections, it does not improve the leading
Coulomb-type splitting rate. In particular, no uniform global
$H^2\to L^2$ convergence rate better than $1/4$ can hold in general.

The theoretical results are further supported by the numerical experiments. The global experiments show that, as the spatial truncation level increases, the singular $L^{-1/4}$ behavior persists over a longer range before the eventual
finite-dimensional first-order regime appears. The local experiments show that
rough data near the $H^2$ threshold exhibit the predicted $t^{5/4}$ one-step
scaling, and the comparisons with Coulomb and regularized Gaussian potentials indicate that the order reduction is tied to the singularity at the origin rather than to the discretization itself. The $\mu$-dependent experiments also show that stronger screening reduces the error prefactor, while leaving the leading Coulomb-type singular scaling unchanged.

There are several natural directions for future work. First, it would be interesting to extend the analysis to higher-order splitting methods, where the
exact error representations and singular commutator structures are more complicated. Second, one may investigate more general singular and low-regularity potentials, especially short-range singular potentials for which the local
singularity and the behavior at infinity can be separated
\cite{BurgarthFacchiHahnJohnssonYuasa2024,BurgarthGalkeHahnvanLuijk2023,BeckerGalkeSalzmannLuijk2024,BaoWang2025,BaoMaWang2024,BaoWang2024,BaoWang2024b,Case1950,FrankLandSpector1971,Meetz1964,ShaoSu2023,AlamaBronsard2023}.
A third direction is to develop sharper lower bound theory, including fixed time
many step lower bounds and genuinely many-body lower bounds, where possible
cancellations between different time steps or different interacting pairs must
be controlled. Finally, it would be useful to connect the continuum estimates proved here with fully discrete numerical schemes by quantifying the interaction
between spatial discretization error and splitting error. We expect these ideas to be useful for singular short-range potentials and, more broadly, for splitting methods for unbounded Schr\"odinger operators.

\section*{Acknowledgments}
This work is supported by the NSF CAREER award DMS-2438074 (D.F. and J.Z.), and the U.S. Department of Energy, Office of Science, Accelerated Research in Quantum Computing Centers, Quantum Utility through Advanced Computational Quantum Algorithms, grant no. DE-SC0025572. The authors also thank the anonymous referee of the manuscript~\cite{FangWuSoffer2025} for suggesting the study of the Yukawa potential as an open question of importance in physical and chemical applications.

\bibliographystyle{unsrt}
\bibliography{yukawa}

\newpage
\appendix
\label{app:app}

\section{Proof of \cref{cor:no-better-rate}}\label{app:cor4}
\begin{proof}
Assume \cref{eq:forbidden-global-bound} held with some $\alpha>1/4$. Taking $L=1$ and $T=t=s$ gives
\begin{equation}
\norm{E_\mu(s)}_{H^2\to L^2}\le C s^{1+\alpha}.
\end{equation}
This contradicts Theorem~\ref{thm:main-one-body-yukawa-lower}, because $1+\alpha>5/4$.
\end{proof}
\section{Auxiliary estimates for the Yukawa Hamiltonian}
\label{app:yukawa-auxiliary}
\begin{lemma}[One-body Yukawa domain and $H^2$ flow bound]
\label[lemma]{lem:one-body-yukawa-H2-flow}
Let $H_\mu=-\Delta+V_\mu$ be the one-body Yukawa Hamiltonian as defined in \cref{eq:yukawa_nbody_potential}, where we fix $\mu>0$.
Then $H_\mu$ is self-adjoint on $H^2(\RR^3)$. Moreover, the graph norm
of $H_\mu$ is equivalent to the $H^2$ norm uniformly:
there exists $C_{\mu}>0$ such that
\begin{equation}
C_{\mu}^{-1}\|\psi\|_{H^2}
\le
\|H_\mu\psi\|_{L^2}+\|\psi\|_{L^2}
\le
C_{\mu}\|\psi\|_{H^2},
\quad
\psi\in H^2(\mathbb R^3).
\label{eq:one-body-yukawa-graph-norm}
\end{equation}
Consequently,
\begin{equation}
\|e^{-\I t H_\mu}\psi\|_{H^2}
\le
C_{\mu}\|\psi\|_{H^2},
\quad
t\in\RR,
\quad
\psi\in H^2(\RR^3).
\label{eq:one-body-yukawa-H2-flow}
\end{equation}
Moreover, $e^{-\I t H_\mu}$ is strongly continuous on
$H^2(\RR^3)$.
\end{lemma}

\begin{proof}
Since $|V_\mu(x)|\le |x|^{-1}$, Hardy's inequality and interpolation give, for every $\varepsilon>0$,
\begin{equation}
\|V_\mu\psi\|_{L^2}
\le\left\|\frac{\psi}{|x|}\right\|_{L^2}
\le2\|\nabla\psi\|_{L^2}
\le\varepsilon\|-\Delta\psi\|_{L^2}
+C_\varepsilon\|\psi\|_{L^2}.
\end{equation}
Hence $V_\mu$ is infinitesimally $-\Delta$-bounded, uniformly for $\mu>0$.
By the Kato-Rellich theorem, $H_\mu=-\Delta+V_\mu$ is self-adjoint on
$H^2(\mathbb R^3)$.
The upper graph norm bound follows from $H_\mu=-\Delta+V_\mu$ and the preceding infinitesimal bound. Conversely, for fixed sufficiently small $\varepsilon>0$,
\begin{equation}
\|-\Delta\psi\|_{L^2}
\le\|H_\mu\psi\|_{L^2}
+\|V_\mu\psi\|_{L^2}
\le\|H_\mu\psi\|_{L^2}
+\varepsilon\|-\Delta\psi\|_{L^2}
+C_\varepsilon\|\psi\|_{L^2}.
\end{equation}
Absorbing the $\varepsilon\|-\Delta\psi\|_{L^2}$ term gives the lower graph norm bound.

Finally, by the functional calculus for the self-adjoint operator $H_\mu$,
\begin{equation}
\|e^{-\I t H_\mu}\psi\|_{L^2}
=
\|\psi\|_{L^2},
\quad
\|H_\mu e^{-\I t H_\mu}\psi\|_{L^2}
=
\|H_\mu\psi\|_{L^2}.
\end{equation}
Using the graph norm equivalence, we obtain
\begin{equation}
\|e^{-\I t H_\mu}\psi\|_{H^2}
\le
C_{\mu}
\left(
\|H_\mu e^{-\I t H_\mu}\psi\|_{L^2}
+
\|e^{-\I t H_\mu}\psi\|_{L^2}
\right)
=
C_{\mu}
\left(
\|H_\mu\psi\|_{L^2}
+
\|\psi\|_{L^2}
\right)
\le
C_{\mu}\|\psi\|_{H^2}.
\end{equation}
This proves \cref{eq:one-body-yukawa-H2-flow}.
Finally, since the graph norm of $H_\mu$ is equivalent to the $H^2$ norm, strong continuity on $H^2$ follows from strong continuity in the graph norm. Indeed, for $\psi\in H^2(\RR^3)$,
\begin{equation}
\|e^{-\I tH_\mu}\psi-\psi\|_{L^2}\to0,
\quad
\|H_\mu(e^{-\I tH_\mu}\psi-\psi)\|_{L^2}
=\|(e^{-\I tH_\mu}-1)H_\mu\psi\|_{L^2}\to0.
\end{equation}
Thus $\|e^{-\I tH_\mu}\psi-\psi\|_{H^2}\to0$.
\end{proof}

\section{Auxiliary estimates for the proof of \cref{lem:pairwise-yukawa-commutator}}
\label[appendix]{app:pairwise-yukawa-commutator}

In this appendix, we establish the two auxiliary estimates \cref{lem:yukawa-relative-derivative-estimate} and  \cref{lem:yukawa-cutoff-estimates} used in the proof of \cref{lem:pairwise-yukawa-commutator}. The argument is the Yukawa analogue of the pairwise Coulomb cutoff estimate. The key ingredient is the set of cutoff estimates for $V_\mu(r)=e^{-\mu r}/r$, with constants depending on the fixed $\mu$. 

\begin{lemma}[Relative-coordinate derivative estimate]
\label[lemma]{lem:yukawa-relative-derivative-estimate}
Let $y,z\in\mathbb R^3$, $p_y=-i\nabla_y$, and $p_z=-i\nabla_z$. 
We write $\partial_{y_j-z_j}:=\frac12(\partial_{y_j}-\partial_{z_j})$ for the directional derivative associated with the relative coordinate $y_j-z_j$ under the unnormalized change of variables $(y,z)\mapsto(y-z,y+z)$.
Then, for $j=1,2,3$ and all
$g\in H^2(\mathbb R^6)$,
\begin{equation}
\left\|
|p_y|\partial_{y_j-z_j}g
\right\|^2
\le
\frac34\left\||p_y|^2g\right\|^2
+\frac14\left\||p_z|^2g\right\|^2 .
\label{eq:yukawa-relative-derivative-estimate}
\end{equation}
Consequently, in the original variables, with
$\partial_{(x_1-x_2)\cdot e_j}:=\frac12(\partial_{x_{1,j}}-\partial_{x_{2,j}})$,
we have
\begin{equation}
\left\|
|p_1|\partial_{(x_1-x_2)\cdot e_j}f
\right\|
\le
\left\||p_1|^2f\right\|
+\left\||p_2|^2f\right\|,
\quad j=1,2,3.
\label{eq:yukawa-relative-derivative-estimate-x}
\end{equation}
\end{lemma}

\begin{proof}
For smooth $g$, integration by parts gives
\begin{equation}
\begin{aligned}
\left\|
|p_y|\partial_{y_j-z_j}g
\right\|^2
&=-\left(
|p_y|g,
\partial_{y_j-z_j}^2 |p_y|g
\right)_{L^2}                          
\le
\frac12
\left(
|p_y|g,
(-\Delta_y-\Delta_z)|p_y|g
\right)_{L^2}                                      \\
&=
\frac12
\left(
\left\||p_y|^2g\right\|^2
+\left\||p_z||p_y|g\right\|^2
\right)                                
\le
\frac34\left\||p_y|^2g\right\|^2
+\frac14\left\||p_z|^2g\right\|^2 .
\end{aligned}
\end{equation}
The result for $g\in H^2$ follows by density. Applying this estimate to the
relative and center-of-mass variables associated with $x_1$ and $x_2$ gives
\eqref{eq:yukawa-relative-derivative-estimate-x}.
\end{proof}

We next introduce the smooth cutoff decomposition for the Yukawa interaction. Let
$0<\beta<1$, set $a=s^\beta$, and choose
$\chi_{\mathrm{sin}}\in C^\infty([0,\infty))$ such that
\begin{equation}
0\le \chi_{\mathrm{sin}}\le1,\quad
\chi_{\mathrm{sin}}(\rho)=1\quad \text{for }0\le \rho\le \frac12,
\quad
\chi_{\mathrm{sin}}(\rho)=0\quad \text{for }\rho\ge1.
\end{equation}
Let $\chi_{\mathrm{reg}}(\rho):=1-\chi_{\mathrm{sin}}(\rho)$.
For
$V_\mu(y):=e^{-\mu |y|}/|y|,
\ y\in\mathbb R^3\setminus\{0\}$,
define
\begin{equation}
v_{\mu,\mathrm{sin}}(y,s)
:=
\chi_{\mathrm{sin}}\left(\frac{|y|}{a}\right)V_\mu(y),
\quad
v_{\mu,\mathrm{reg}}(y,s)
:=\chi_{\mathrm{reg}}\left(\frac{|y|}{a}\right)V_\mu(y).
\label{eq:yukawa-pair-cutoff-decomposition}
\end{equation}

\begin{proof}[Proof of \cref{lem:yukawa-cutoff-estimates}]
Write $r=|y|$. Since
\begin{equation}
V_\mu(r)=\frac{e^{-\mu r}}{r},
\quad
V_\mu'(r)
=-\left(\mu+\frac1r\right)V_\mu(r),
\end{equation}
and, away from $r=0$,
$\Delta V_\mu=\mu^2V_\mu$.
Although distributionally one has
$\Delta V_\mu=\mu^2V_\mu+4\pi\delta_0$, the delta contribution does not appear
in the following calculation because $v_{\mu,\mathrm{reg}}$ vanishes in a
neighborhood of $r=0$. We compute
\begin{equation}
\Delta_y v_{\mu,\mathrm{reg}}(y,s)
=
\chi_{\mathrm{reg}}\left(\frac r a\right)\mu^2V_\mu(r)
+2a^{-1}\chi_{\mathrm{reg}}'\left(\frac r a\right)V_\mu'(r) +
\left[
a^{-2}\chi_{\mathrm{reg}}''\left(\frac r a\right)
+
2a^{-1}r^{-1}\chi_{\mathrm{reg}}'\left(\frac r a\right)
\right]V_\mu(r).
\label{eq:yukawa-Delta-vreg-expansion}
\end{equation}
The terms involving derivatives of $\chi_{\mathrm{reg}}$ are supported on
$\{a/2\le r\le a\}$. On this annulus $r\sim a$, and hence
\begin{equation}
\left|
2a^{-1}\chi_{\mathrm{reg}}'\left(\frac r a\right)V_\mu'(r)
+
\left[
a^{-2}\chi_{\mathrm{reg}}''\left(\frac r a\right)
+
2a^{-1}r^{-1}\chi_{\mathrm{reg}}'\left(\frac r a\right)
\right]V_\mu(r)
\right|
\le
C_{\mu}\left(a^{-3}+\mu a^{-2}\right)
\mathbf 1_{\{a/2\le r\le a\}} .
\end{equation}
Therefore,
\begin{equation}
\left\|
\left(a^{-3}+\mu a^{-2}\right)
\mathbf 1_{\{a/2\le |y|\le a\}}
\right\|_{L^2_y}
\le C_{\mu}
\left(
a^{-3/2}+\mu a^{-1/2}
\right)
\le C_{\mu}a^{-3/2},
\end{equation}
because $a\le1$ and $\mu\le\mu_0$. For the remaining Yukawa bulk term,
\begin{equation}
\|\mu^2 V_\mu\|_{L^2(\mathbb R^3)}^2
=4\pi\mu^4\int_0^\infty e^{-2\mu r}\,dr
=2\pi\mu^3,
\end{equation}
and therefore
$\|\mu^2 V_\mu\|_{L^2}
\le C\mu^{3/2}
\le C_{\mu}a^{-3/2}$.
This proves \cref{eq:yukawa-cutoff-laplacian-bound}.
Next,
\begin{equation}
\nabla_y v_{\mu,\mathrm{reg}}(y,s)
=\chi_{\mathrm{reg}}\left(\frac r a\right)\nabla_y V_\mu(y)
+a^{-1}\chi_{\mathrm{reg}}'\left(\frac r a\right)\frac{y}{r}V_\mu(r).
\end{equation}
Since
\begin{equation}
|\nabla_y V_\mu(y)|
=\left(\mu+\frac1r\right)V_\mu(r),
\end{equation}
we have on the support of $\chi_{\mathrm{reg}}$,
\begin{equation}
r|\nabla_y V_\mu(y)|
=\mu e^{-\mu r}
+\frac{e^{-\mu r}}{r}
\le \mu+\frac{2}{a}
\le C_{\mu}a^{-1}.
\end{equation}
For the cutoff derivative term, the support of $\chi_{\mathrm{reg}}'$ lies in
$\{a/2\le r\le a\}$, and hence
\begin{equation}
r\,a^{-1}
\left|\chi_{\mathrm{reg}}'\left(\frac r a\right)\right|
V_\mu(r)
\le
Ca^{-1}.
\end{equation}
This proves \cref{eq:yukawa-cutoff-gradient-bound}.
Finally, since $0\le V_\mu(y)\le |y|^{-1}$ and $v_{\mu,\mathrm{sin}}$ is supported in $\{|y|\le a\}$,
\begin{equation}
\|V_{\mu,\mathrm{sin}}(\cdot,s)\|_{L^2}^2
\le 4\pi\int_0^a dr
\le Ca.
\end{equation}
Thus $\|V_{\mu,\mathrm{sin}}(\cdot,s)\|_{L^2}\le Ca^{1/2}=Cs^{\beta/2}$, proving \cref{eq:yukawa-cutoff-singular-bound}.
\end{proof}

\section{Proof of the One-body Yukawa upper bound}\label{app:one_upper}

\begin{proof}[Proof of \cref{lem:one-body-yukawa-commutator-upper}]
It suffices to prove the estimate for $f,g\in \mathcal S(\mathbb R^3)$.
The general case then follows by density. Throughout the proof, $C_{\mu}$
denotes a constant depending only on $\mu$ and on the fixed cutoff function.
Let $0<\beta<1$ be chosen later and set $a=s^\beta$. Choose $\chi_{\mathrm{sin}}\in C^\infty([0,\infty))$ such that
\begin{equation}
0\le \chi_{\mathrm{sin}}\le 1,\quad
\chi_{\mathrm{sin}}(\rho)=1\quad \text{for }0\le \rho\le \frac12,
\quad
\chi_{\mathrm{sin}}(\rho)=0\quad \text{for }\rho\ge 1.
\end{equation}
Let $\chi_{\mathrm{reg}}(\rho):=1-\chi_{\mathrm{sin}}(\rho)$, we decompose
$V_\mu=V_{\mu,\mathrm{sin}}+V_{\mu,\mathrm{reg}}$,
where
\begin{equation}
V_{\mu,\mathrm{sin}}(x,s)
=\chi_{\mathrm{sin}}\left(\frac{|x|}{a}\right)V_\mu(x),
\quad
V_{\mu,\mathrm{reg}}(x,s)
=\chi_{\mathrm{reg}}\left(\frac{|x|}{a}\right)V_\mu(x).
\end{equation}
We first estimate the singular part. Since $|V_\mu(x)|\le |x|^{-1}$ and $V_{\mu,\mathrm{sin}}$ is supported in $\{|x|\le a\}$, we have
\begin{equation}
\|V_{\mu,\mathrm{sin}}\|_{L^2}^2
\le
4\pi\int_0^a \frac{1}{r^2}r^2\,dr
\le Ca.
\end{equation}
Hence
$\|V_{\mu,\mathrm{sin}}\|_{L^2}
\le Ca^{1/2}
=Cs^{\beta/2}$.
Therefore, using
\begin{equation}
[e^{is\Delta},V_{\mu,\mathrm{sin}}]f
=e^{is\Delta}(V_{\mu,\mathrm{sin}}f)
-V_{\mu,\mathrm{sin}}e^{is\Delta}f,
\end{equation}
the unitarity of $e^{is\Delta}$ on $L^2$, and the Sobolev embedding
$H^2(\RR^3)\hookrightarrow L^\infty(\RR^3)$, we obtain
\begin{equation}
\begin{aligned}
\left\|
[e^{is\Delta},V_{\mu,\mathrm{sin}}]f
\right\|_{L^2}
&\le
\|V_{\mu,\mathrm{sin}}f\|_{L^2}
+\|V_{\mu,\mathrm{sin}}e^{is\Delta}f\|_{L^2}        \\
&\le
\|V_{\mu,\mathrm{sin}}\|_{L^2}
\left(
\|f\|_{L^\infty}
+\|e^{is\Delta}f\|_{L^\infty}
\right)                                
\le
Cs^{\beta/2}\|f\|_{H^2}.
\end{aligned}
\label{eq:one-body-singular-commutator-bound}
\end{equation}
We now estimate the regular part. Duhamel's commutator formula gives
\begin{equation}
[e^{is\Delta},V_{\mu,\mathrm{reg}}]
=i\int_0^s
e^{i\tau\Delta}
[\Delta,V_{\mu,\mathrm{reg}}]
e^{i(s-\tau)\Delta}
\,d\tau .
\label{eq:one-body-regular-duhamel}
\end{equation}
Let
$h_\tau=e^{-i\tau\Delta}g$, and $u_\tau=e^{i(s-\tau)\Delta}f$.
Then
$[\Delta,V_{\mu,\mathrm{reg}}]u_\tau
=(\Delta V_{\mu,\mathrm{reg}})u_\tau
+2\nabla V_{\mu,\mathrm{reg}}\cdot \nabla u_\tau$.
We first control the zeroth-order term. Since
\begin{equation}
V_\mu(r)=\frac{e^{-\mu r}}{r},
\quad
\partial_r V_\mu(r)
=-\left(\mu+\frac1r\right)V_\mu(r),
\end{equation}
and, away from $r=0$,
$\Delta V_\mu=\mu^2 V_\mu$,
so we may compute $\Delta V_{\mu,\mathrm{reg}}$ without any distributional
$\delta_0$ contribution, because $\chi_{\mathrm{reg}}(|x|/a)$ vanishes in a
neighborhood of the origin. Writing $r=|x|$, we have
\begin{equation}
\Delta V_{\mu,\mathrm{reg}}
=\chi_{\mathrm{reg}}\left(\frac r a\right)\mu^2 V_\mu
+2a^{-1}\chi_{\mathrm{reg}}'\left(\frac r a\right)\partial_r V_\mu
+\left[
a^{-2}\chi_{\mathrm{reg}}''\left(\frac r a\right)
+
2a^{-1}r^{-1}\chi_{\mathrm{reg}}'\left(\frac r a\right)
\right]V_\mu .
\label{eq:one-body-delta-vreg-expansion}
\end{equation}
The terms involving derivatives of $\chi_{\mathrm{reg}}$ are supported on the
annulus $\{a/2\le r\le a\}$. On this annulus $r\sim a$, and hence
\begin{equation}
\left|
2a^{-1}\chi_{\mathrm{reg}}'\left(\frac r a\right)\partial_r V_\mu
+
\left[
a^{-2}\chi_{\mathrm{reg}}''\left(\frac r a\right)
+
2a^{-1}r^{-1}\chi_{\mathrm{reg}}'\left(\frac r a\right)
\right]V_\mu
\right|
\le
C\left(a^{-3}+\mu a^{-2}\right)
\mathbf 1_{\{a/2\le r\le a\}} .
\end{equation}
Therefore,
\begin{equation}
\left\|
\left(a^{-3}+\mu a^{-2}\right)
\mathbf 1_{\{a/2\le |x|\le a\}}
\right\|_{L^2}
\le
C\left(a^{-3/2}+\mu a^{-1/2}\right).
\end{equation}
For the remaining term, we use
\begin{equation}
\|\mu^2 V_\mu\|_{L^2}^2
=
4\pi \mu^4\int_0^\infty e^{-2\mu r}\,dr
=
2\pi \mu^3,
\end{equation}
and hence
$\|\mu^2 V_\mu\|_{L^2}
\le C\mu^{3/2}$.
Combining these bounds and recalling $a=s^\beta$, we obtain
\begin{equation}
\|\Delta V_{\mu,\mathrm{reg}}\|_{L^2}
\le
C_{\mu}
\left(
s^{-3\beta/2}
+\mu s^{-\beta/2}
+\mu^{3/2}
\right).
\label{eq:one-body-delta-vreg-L2-bound}
\end{equation}
Thus, using the unitarity of $e^{-i\tau\Delta}$ on $L^2$ and the Sobolev
embedding $H^2(\mathbb R^3)\hookrightarrow L^\infty(\mathbb R^3)$, we get
\begin{equation}
\begin{aligned}
\int_0^s
\left|
\left\langle
h_\tau,
(\Delta V_{\mu,\mathrm{reg}})u_\tau
\right\rangle
\right|
\,d\tau
&\le
\int_0^s
\|h_\tau\|_{L^2}
\|\Delta V_{\mu,\mathrm{reg}}\|_{L^2}
\|u_\tau\|_{L^\infty}
\,d\tau  \\
&\le
C_{\mu}
\left(
s^{1-3\beta/2}
+
\mu s^{1-\beta/2}
+
\mu^{3/2}s
\right)
\|g\|_{L^2}\|f\|_{H^2}.
\end{aligned}
\label{eq:one-body-zero-order-regular-bound}
\end{equation}
It remains to estimate the first-order term. We first record a pointwise bound for
$\nabla V_{\mu,\mathrm{reg}}$. Since
\begin{equation}
\nabla V_{\mu,\mathrm{reg}}
=\chi_{\mathrm{reg}}\left(\frac r a\right)\nabla V_\mu
+a^{-1}\chi_{\mathrm{reg}}'\left(\frac r a\right)\frac{x}{r}V_\mu,
\end{equation}
and since the derivative of $\chi_{\mathrm{reg}}$ is supported where $r\sim a$,
we have
$a^{-1}
\left|\chi_{\mathrm{reg}}'\left(r /a\right)\right|
\le C/r$.
Moreover,
$|\nabla V_\mu|
=\left(\mu+1/r\right)V_\mu
\le
C_{\mu}\left(1+1/r\right)V_\mu$.
Hence
\begin{equation}
|\nabla V_{\mu,\mathrm{reg}}(x,s)|
\le C_{\mu}V_\mu(x)
\left(
\frac{1}{|x|}+1
\right).
\label{eq:one-body-gradient-vreg-pointwise}
\end{equation}
We now use Kato smoothing. Since $0\le V_\mu(x)\le |x|^{-1}$, the standard free Kato smoothing estimate for the Coulomb weight implies
\begin{equation}
\left(
\int_0^s
\|V_\mu e^{-i\tau\Delta}g\|_{L^2}^2
\,d\tau
\right)^{1/2}
\le\left(
\int_0^s
\left\|\frac1{|x|}e^{-i\tau\Delta}g\right\|_{L^2}^2
\,d\tau
\right)^{1/2}
\le C\|g\|_{L^2}.
\label{eq:one-body-kato-smoothing}
\end{equation}
Using \eqref{eq:one-body-gradient-vreg-pointwise}, we obtain
\begin{equation}
\begin{aligned}
\int_0^s
\left|
\left\langle
h_\tau,
2\nabla V_{\mu,\mathrm{reg}}\cdot\nabla u_\tau
\right\rangle
\right|
\,d\tau
&\le C_{\mu}
\int_0^s
\|V_\mu h_\tau\|_{L^2}
\left(
\left\|\frac{\nabla u_\tau}{|x|}\right\|_{L^2}
+\|\nabla u_\tau\|_{L^2}
\right)
\,d\tau .
\end{aligned}
\end{equation}
By Hardy's inequality applied to each component $\partial_j u_\tau$, we have
\begin{equation}
\left\|\frac{\nabla u_\tau}{|x|}\right\|_{L^2}
\le
C\|u_\tau\|_{H^2}.
\end{equation}
Also,
$\|\nabla u_\tau\|_{L^2}
\le \|u_\tau\|_{H^2}$.
Since the free Schr\"odinger group preserves the $H^2$ norm,
$\|u_\tau\|_{H^2}
=\|e^{i(s-\tau)\Delta}f\|_{H^2}
=\|f\|_{H^2}$.
Therefore, by Cauchy's inequality in the time variable,
\begin{equation}
\int_0^s
\left|
\left\langle
h_\tau,
2\nabla V_{\mu,\mathrm{reg}}\cdot\nabla u_\tau
\right\rangle
\right|
\,d\tau
\le
C_{\mu}\|f\|_{H^2}
\int_0^s
\|V_\mu e^{-i\tau\Delta}g\|_{L^2}
\,d\tau  
\le
C_{\mu}s^{1/2}
\|g\|_{L^2}\|f\|_{H^2}.
\label{eq:one-body-first-order-regular-bound}
\end{equation}
Combining the singular estimate \cref{eq:one-body-singular-commutator-bound},
the zeroth-order regular estimate \cref{eq:one-body-zero-order-regular-bound},
and the first-order Kato estimate \cref{eq:one-body-first-order-regular-bound},
we obtain
\begin{equation}
\left|
\left\langle
g,
[e^{is\Delta},V_\mu]f
\right\rangle
\right|
\le C_{\mu}
\left(
s^{\beta/2}
+s^{1-3\beta/2}
+\mu s^{1-\beta/2}
+\mu^{3/2}s
+s^{1/2}
\right)
\|g\|_{L^2}\|f\|_{H^2}.
\end{equation}
Choosing $\beta=1/2$ gives
\begin{equation}
s^{\beta/2}=s^{1/4},
\quad
s^{1-3\beta/2}=s^{1/4}.
\end{equation}
Hence
\begin{equation}
\left|
\left\langle
g,
[e^{is\Delta},V_\mu]f
\right\rangle
\right|
\le
C_{\mu}
\left(
s^{1/4}
+\mu s^{3/4}
+\mu^{3/2}s
+s^{1/2}
\right)
\|g\|_{L^2}\|f\|_{H^2}.
\end{equation}
Since $0<s\le 1$ and $0<\mu\le \mu_0$, the last three terms are bounded by
$C_{\mu}s^{1/4}$. Therefore,
\begin{equation}
\left|
\left\langle
g,
[e^{is\Delta},V_\mu]f
\right\rangle
\right|
\le
C_{\mu}s^{1/4}
\|g\|_{L^2}\|f\|_{H^2}.
\end{equation}
Taking the supremum over $\|g\|_{L^2}\le1$ and $\|f\|_{H^2}\le1$ proves the
claim.
\end{proof}

\begin{thm}[One-body Yukawa upper bound]
\label{cor:one-body-yukawa-upper-bound}
Fix $\mu>0$. Let $H_\mu=-\Delta+V_\mu$ be the one-body Yukawa Hamiltonian defined by the $N=1$ case of \cref{eq:yukawa_nbody_potential}. 
Let $U_{1,\mu}(t)$ be the
first-order splitting operator defined in \cref{eq:yukawa_nbody_trotter}.
Then, for every $\psi_0\in H^2(\mathbb R^3)$, $T>0$, $L\in\mathbb N$, and $t=T/L\in(0,1]$,
\begin{equation}
\left\|
\left(e^{-itV_\mu}e^{it\Delta}\right)^L\psi_0
-e^{-iTH_\mu}\psi_0
\right\|_{L^2(\mathbb R^3)}
\le
C_{\mu}Tt^{1/4}\|\psi_0\|_{H^2(\mathbb R^3)}.
\label{eq:one-body-yukawa-long-time-upper-bound}
\end{equation}
\end{thm}

\begin{proof}
The local error representation gives
\begin{equation}
e^{-itV_\mu}e^{it\Delta}\psi_0-e^{-itH_\mu}\psi_0
=i\int_0^t
e^{-isV_\mu}
[e^{is\Delta},V_\mu]
e^{-i(t-s)H_\mu}\psi_0\,ds .
\label{eq:one-body-local-error-representation-upper}
\end{equation}
By \cref{lem:one-body-yukawa-H2-flow}, there exists $C_{\mu}>0$ such that,
uniformly for $\mu>0$,
\begin{equation}
\|e^{-i\tau H_\mu}\psi_0\|_{H^2}
\le
C_{\mu}\|\psi_0\|_{H^2},
\quad
\tau\in\mathbb R.
\label{eq:one-body-H2-flow-bound}
\end{equation}
Using \cref{lem:one-body-yukawa-commutator-upper} in
\cref{eq:one-body-local-error-representation-upper}, together with
\cref{eq:one-body-H2-flow-bound}, yields
\begin{equation}
\left\|
e^{-itV_\mu}e^{it\Delta}\psi_0-e^{-itH_\mu}\psi_0
\right\|_{L^2}
\le
C_{\mu}\|\psi_0\|_{H^2}
\int_0^t s^{1/4}\,ds  
=
\frac45 C_{\mu}t^{5/4}\|\psi_0\|_{H^2} 
\le
C_{\mu}t^{5/4}\|\psi_0\|_{H^2}.
\end{equation}
The long-time estimate follows from the standard telescoping identity. Let
$U_{1,\mu}(t)=e^{-itV_\mu}e^{it\Delta}$ and $U_\mu(t)=e^{-itH_\mu}$. Then
\begin{equation}
U_{1,\mu}(t)^L-U_\mu(t)^L
=\sum_{\ell=0}^{L-1}
U_{1,\mu}(t)^{L-1-\ell}
\left(U_{1,\mu}(t)-U_\mu(t)\right)
U_\mu(t)^\ell .
\end{equation}
Since $U_{1,\mu}(t)$ and $U_\mu(t)$ are unitary on $L^2$, using the local
bound with shifted initial data $U_\mu(t)^\ell\psi_0$, together with
\cref{eq:one-body-H2-flow-bound}, gives
\begin{equation}
\left\|
U_{1,\mu}(t)^L\psi_0-U_\mu(t)^L\psi_0
\right\|_{L^2}
\le
LC_{\mu}t^{5/4}\|\psi_0\|_{H^2}.
\end{equation}
Since $L=T/t$, we obtain
$\left\|
\left(e^{-itV_\mu}e^{it\Delta}\right)^L\psi_0
-e^{-iTH_\mu}\psi_0
\right\|_{L^2}
\le C_{\mu}Tt^{1/4}\|\psi_0\|_{H^2}$.
This proves \cref{eq:one-body-yukawa-long-time-upper-bound}.
\end{proof}

\section{Many-body Sobolev flow estimate}
\label{sec:yukawa-many-body-flow-estimate}
It remains to control the Sobolev norms of the exact Yukawa evolution appearing on the right-hand side of \cref{eq:yukawa-local-error-sup-flow}. For every $\psi_0\in D(H_\mu)$, both $\|\psi(t)\|$ and $\|H_\mu\psi(t)\|$ are conserved along the exact flow. However, converting this Hamiltonian control into Sobolev control introduces dependence on the particle number $N$. We therefore use the many-body $H^2$ flow strategy from the Coulomb case \cite[Theorem~7]{FangWuSoffer2025}, while keeping the $N$-dependence explicit. The key ingredient is the many-body HLS bound
\begin{equation}
\left\|V_\mu |p|^{-1}\right\|_{L^2\to L^2}
\le 2C_{\mathrm{int}}N^{3/2},
\label{eq:yukawa-HLS-bound-short}
\end{equation}
where $|p|=(-\Delta)^{1/2}$. This follows from the Coulomb many-body HLS estimate, since each Yukawa pair kernel is dominated by the corresponding Coulomb kernel and $|c_{jk}|\le C_{\mathrm{int}}$.

\begin{thm}[Many-body Yukawa $H^2$ flow estimate]
\label{prop:yukawa-H2-flow}
Let $\psi(t)=e^{-itH_\mu}\psi_0$, with $\psi_0\in H^2(\mathbb R^{3N})$. Then, for every $t\in\mathbb R$,
\begin{equation}
\|\psi(t)\|_{H^2}
\le
C\left(1+C_{\mathrm{int}}N^{3/2}
+C_{\mathrm{int}}^2N^3\right)
\|\psi_0\|_{H^2}.
\label{eq:yukawa-H2-flow}
\end{equation}
where $C>0$ is a universal constant. In particular, for fixed $C_{\mathrm{int}}$, the right-hand side is
$O_{C_{\mathrm{int}}}(N^3)\|\psi_0\|_{H^2}$.
\end{thm}

\begin{proof}
We first justify that the quantities appearing in the estimate are well-defined.
By \cref{eq:yukawa-HLS-bound-short}, for every $\phi\in H^2(\mathbb R^{3N})$,
\begin{equation}
\|V_\mu\phi\|
\le \|V_\mu |p|^{-1}\|\,\||p|\phi\|
\le C_{\mathrm{int}}N^{3/2}\||p|\phi\|.
\end{equation}
By the standard interpolation estimate, for every $\eta>0$,
$\||p|\phi\|
\le
\eta\|-\Delta\phi\|
+C_\eta\|\phi\|$.
Combining this with the preceding bound and choosing $\eta$ sufficiently small gives, for every $\varepsilon>0$,
\begin{equation}
\|V_\mu\phi\|
\le
\varepsilon\|-\Delta\phi\|
+C_{\varepsilon,N,C_{\mathrm{int}}}\|\phi\|.
\end{equation}
Hence $V_\mu$ is infinitesimally $-\Delta$-bounded. Since $V_\mu$ is real-valued, the Kato-Rellich theorem implies that $H_\mu=-\Delta+V_\mu$ is self-adjoint with domain $D(H_\mu)=H^2(\RR^{3N})$. The same infinitesimal bound gives equivalence of the graph norm of $H_\mu$ and the $H^2$ norm. Therefore, by functional calculus, $e^{-itH_\mu}D(H_\mu)=D(H_\mu)$, and $H_\mu e^{-itH_\mu}\psi_0=e^{-itH_\mu}H_\mu\psi_0$.
In particular, $e^{-itH_\mu}$ preserves $H^2(\mathbb R^{3N})$, and, using $H_\mu=-\Delta+V_\mu$, we obtain
\begin{equation}
(-\Delta)e^{-itH_\mu}\psi_0
=e^{-itH_\mu}H_\mu\psi_0
-V_\mu e^{-itH_\mu}\psi_0
\label{eq:yukawa-delta-identity-in-flow-proof}
\end{equation}
in $L^2$.
With this domain justification in place, we use the many-body Coulomb $H^2$ flow argument \cite[Theorem~7]{FangWuSoffer2025}. One of the many-body ingredients in that argument is the HLS-type bound for $V|p|^{-1}$, proved in the Coulomb case in \cite[Lemma~8]{FangWuSoffer2025}; in the Yukawa case, the corresponding bound is precisely \cref{eq:yukawa-HLS-bound-short}. Since $V_\mu$ is real-valued, the adjoint of the bounded operator $V_\mu|p|^{-1}$ is the bounded extension of $|p|^{-1}V_\mu$. Hence
\begin{equation}
\bigl\||p|^{-1}V_\mu\bigr\|
=\bigl\|V_\mu|p|^{-1}\bigr\|
\le
2C_{\mathrm{int}}N^{3/2}.
\label{eq:yukawa-adjoint-HLS-bound}
\end{equation}
Substituting \cref{eq:yukawa-HLS-bound-short,eq:yukawa-adjoint-HLS-bound} into the Coulomb $H^2$ flow argument \cite[Proof of Theorem~7]{FangWuSoffer2025} yields
\begin{equation}
\|(-\Delta)\psi(t)\|
\le
C\left(1+C_{\mathrm{int}}N^{3/2}
+C_{\mathrm{int}}^2N^3\right)
\|\psi_0\|_{H^2},
\label{eq:yukawa-laplacian-flow}
\end{equation}
with a universal constant $C>0$.
More explicitly, the remaining part of the argument is the same low and high frequency decomposition used in the Coulomb case to control $\||p|e^{-irH_\mu}\psi_0\|$; no new Yukawa-specific estimate is needed after replacing the Coulomb HLS bound by \cref{eq:yukawa-HLS-bound-short}.
Since the $L^2$ norm is conserved under the exact Yukawa flow $\|\psi(t)\|=\|\psi_0\|$,
we obtain the desired result \cref{eq:yukawa-H2-flow}.
\end{proof}

We also use the elementary momentum summation bound \cite[Lemma~16]{FangWuSoffer2025}
\begin{equation}
\sum_{1\le a<b\le N}
\left(
\|\langle p_a\rangle^2 f\|
+\|\langle p_b\rangle^2 f\|
\right)
\le (N-1)N^{3/2}\|f\|
+(N-1)N^{1/2}\|(-\Delta)f\|.
\label{eq:yukawa-momentum-sum-inline}
\end{equation}
This is the same bookkeeping estimate as in the Coulomb argument, following from Cauchy-Schwarz and $\sum_{a=1}^N |p_a|^2=|p|^2$.
Combining \cref{eq:yukawa-momentum-sum-inline} with \cref{eq:yukawa-laplacian-flow} and conservation of the $L^2$ norm gives, for
all $t\in\mathbb R$,
\begin{equation}
\begin{aligned}
\sum_{1\le a<b\le N}
&\left(
\|\langle p_a\rangle^2\psi(t)\|
+\|\langle p_b\rangle^2\psi(t)\|
\right) 
\le
(N-1)N^{3/2}\|\psi(t)\|
+(N-1)N^{1/2}\|(-\Delta)\psi(t)\| \\
&\qquad\le
\Bigl[
(N-1)N^{3/2}
+ C (N-1)N^{1/2}
\bigl(1+C_{\mathrm{int}}N^{3/2}+C_{\mathrm{int}}^2N^3\bigr)
\Bigr]
\|\psi_0\|_{H^2},
\end{aligned}
\label{eq:yukawa-sobolev-sum-flow}
\end{equation}
where $C>0$ is the universal constant from \cref{eq:yukawa-laplacian-flow}. In particular, the right-hand side is
$O_{C_{\mathrm{int}}}(N^{9/2})\|\psi_0\|_{H^2}$.

\section{General coupling lower bound}
\label{app:general-coupling-lower}
\begin{thm}[One-body short-time lower bound for general Yukawa coupling]
\label{thm:general-one-body-yukawa-lower}
Fix $0<\mu\le\mu_0$ and $Z\in\RR\setminus\{0\}$, and let $H_{\mu,Z}$ be defined in \cref{eq:general_Ham}. 
Define the first-order splitting operator as
\begin{equation}
U_{1,\mu,Z}(t)
:=e^{-itV_{\mu,Z}}e^{it\Delta}.
\label{eq:general-coupling-one-body-splitting}
\end{equation}
Then there exist $\psi_\ast\in C_c^\infty(\RR^3)$ with $\psi_\ast(0)\neq0$ and constants $c_{Z,\psi_\ast}>0$, $t_{0,\mu,Z,\psi_\ast}>0$ such that, for all $0<t<t_{0,\mu,Z,\psi_\ast}$,
\begin{equation}
\left\|
U_{1,\mu,Z}(t)\psi_\ast
-e^{-itH_{\mu,Z}}\psi_\ast
\right\|_{L^2(\RR^3)}
\ge c_{Z,\psi_\ast}t^{5/4}.
\label{eq:main-one-body-yukawa-lower-general}
\end{equation}
\end{thm}

\begin{lemma}[Replacement of the exact Yukawa flow]
\label[lemma]{lem:replace_yukawa_exact_flow_lower_bound}
Let
\begin{equation}
\label{eq:R_flow_yukawa_def}
R_{\mathrm{flow}}^\mu(t)
:=
\I\int_0^t
e^{-\I sV_\mu}
[e^{\I s\Delta},V_\mu]
\left(
e^{-\I(t-s)H_\mu}\psi_\ast-\psi_\ast
\right)
\,ds .
\end{equation}
Then
\begin{equation}
\label{eq:R_flow_yukawa_lower_order}
\|R_{\mathrm{flow}}^\mu(t)\|_{L^2(\RR^3)}
=o(t^{5/4})
\qquad\text{as }t\to0.
\end{equation}
\end{lemma}
\begin{proof}
By \cref{lem:one-body-yukawa-commutator-upper}, we have
\begin{equation}
\left\|
[e^{\I s\Delta},V_\mu]f
\right\|_{L^2}
\le
C_{\mu_0}s^{1/4}\|f\|_{H^2},
\quad 0<s\le1.
\end{equation}
Multiplication by $e^{-\I sV_\mu}$ is unitary on $L^2$. Hence
\begin{equation}
\begin{aligned}
\|R_{\mathrm{flow}}^\mu(t)\|_{L^2}
\le
C_{\mu_0}
\int_0^t
s^{1/4}
\left\|
e^{-\I(t-s)H_\mu}\psi_\ast-\psi_\ast
\right\|_{H^2}
\,ds.
\end{aligned}
\end{equation}
Here, we set
$\eta_\mu(t)
:=\sup_{0\le r\le t}
\left\|
e^{-\I rH_\mu}\psi_\ast-\psi_\ast
\right\|_{H^2}$.
By \cref{lem:one-body-yukawa-H2-flow}, the group $e^{-\I rH_\mu}$ is strongly
continuous on $H^2(\mathbb R^3)$. Therefore $\eta_\mu(t)\to0$ as
$t\to0$. Thus
\begin{equation}
\|R_{\mathrm{flow}}^\mu(t)\|_{L^2}
\le
C_{\mu_0}\eta_\mu(t)
\int_0^t s^{1/4}\,ds
=\frac45 C_{\mu_0}\eta_\mu(t)t^{5/4}
=o(t^{5/4}).
\end{equation}
The rescaled statement follows from
$\left\|
t^{-1/2}R_{\mathrm{flow}}^\mu(t)(\sqrt t\,\cdot)
\right\|_{L^2_y}
=t^{-5/4}
\|R_{\mathrm{flow}}^\mu(t)\|_{L^2_x}$.
\end{proof}

\begin{proof}[Proof of \cref{thm:general-one-body-yukawa-lower}]
Choose $\psi_\ast\in C_c^\infty(\RR^3)$ with $\psi_\ast(0)\neq0$. By the local error representation and the exact flow replacement estimate \cref{lem:replace_yukawa_exact_flow_lower_bound}, the error reduces to the same leading singular form as in the proof of \cref{thm:main-one-body-yukawa-lower}, up to an $o(t^{5/4})$ remainder. The only change is the distributional identity
$\Delta V_{\mu}=\mu^2V_{\mu}-4\pi Z\delta_0$.
Thus the free delta contribution has coefficient $-4\pi Z\psi_\ast(0)$, while the non-delta terms and the phase replacement error remain lower order. Therefore,
\begin{equation}
\left\|U_{1,\mu,Z}(t)\psi_\ast
-e^{-itH_{\mu,Z}}\psi_\ast
\right\|_{L^2(\RR^3)}
\ge 4\pi |Z|\,|\psi_\ast(0)|C_\Delta t^{5/4}
-o(t^{5/4}).
\end{equation}
Since $Z\neq0$ and $\psi_\ast(0)\neq0$, the desired bound follows for all sufficiently small $t$.
\end{proof}
\section{Replacement Lemmas}
\label{app:delta-replacement-proofs}

\begin{lemma}[One-layer free delta estimate]
\label[lemma]{lem:free-delta-primitive-bound}
For $0<s\le1$, define
\begin{equation} \label{eq:G}
G(s,x):=\int_0^s e^{\I\tau\Delta_x}\delta_0(x)\,d\tau .
\end{equation}
Then $G(s,\cdot)\in L^2(\RR^3)$ and
\begin{equation}
\label{eq:free-delta-primitive-bound}
\|G(s,\cdot)\|_{L^2(\mathbb R^3)}
\le
C s^{1/4},
\quad 0<s\le1.
\end{equation}
For later use, if
$L_\theta(y):=\int_0^\theta e^{\I\alpha\Delta_y}\delta_0(y)\,d\alpha$, then
\begin{equation}
\label{eq:L-theta-bound}
\|L_\theta\|_{L^2_y}
\le
C\theta^{1/4},
\quad 0<\theta\le1.
\end{equation}
\end{lemma}
\begin{proof}
Using the same Fourier convention as in \cref{eq:fourier_convention}, and using $\widehat{\delta_0}=1$ and
$\widehat{e^{\I\tau\Delta}f}(\xi)
=e^{-\I\tau|\xi|^2}\widehat f(\xi)$,
we obtain
\begin{equation}
\widehat{G(s)}(\xi)
=\int_0^s e^{-\I\tau|\xi|^2}\,d\tau
=\frac{1-e^{-\I s|\xi|^2}}{\I|\xi|^2},
\end{equation}
where the quotient is understood by continuous extension at $\xi=0$. 
By Plancherel, with the harmless normalization constant absorbed into $C$,
\begin{equation}
\|G(s)\|_{L^2}^2
\le
C\int_{\mathbb R^3}
\frac{|1-e^{-\I s|\xi|^2}|^2}{|\xi|^4}\,d\xi .
\end{equation}
Changing variables $\eta=\sqrt{s}\xi$, we obtain
\begin{equation}
\|G(s)\|_{L^2}^2
\le
Cs^{1/2}
\int_{\mathbb R^3}
\frac{|1-e^{-\I|\eta|^2}|^2}{|\eta|^4}\,d\eta .
\end{equation}
The last integral is finite, because near $\eta=0$,
$|1-e^{-\I|\eta|^2}|\lesssim |\eta|^2$,
while for large $|\eta|$, the integrand is bounded by $C|\eta|^{-4}$,
which is integrable in dimension three. Hence
$\|G(s)\|_{L^2}
\le
Cs^{1/4}$.
The estimate for $L_\theta$ follows from the same argument, with $s=\theta$ and the variable $x$ replaced by $y$.
\end{proof}

\begin{lemma}[Replacement of the Yukawa phase in the leading delta term]
\label[lemma]{lem:replace_yukawa_phase_delta_lower_bound}
Let $G$ be defined as in \cref{lem:free-delta-primitive-bound}. Define
\begin{equation}
R_{\mathrm{phase},\delta}^\mu(t)
:=\int_0^t
\left(e^{-\I sV_\mu}-1\right)
G(s,\cdot)\,ds .
\end{equation}
Then
\begin{equation}
\|R_{\mathrm{phase},\delta}^\mu(t)\|_{L^2}
=o(t^{5/4})
\quad
\text{as }t\to0.
\end{equation}
\end{lemma}

\begin{proof}
Set $s=t\theta$ and $x=\sqrt t\,y$. Since
\begin{equation}
V_\mu(\sqrt t\,y)
=-t^{-1/2}
\frac{e^{-\mu\sqrt t|y|}}{|y|},
\end{equation}
we define
$m^\mu_{t,\theta}(y)
:=e^{-\I t\theta V_\mu(\sqrt t\,y)}
=\exp\left(
\I\theta\sqrt t
\frac{e^{-\mu\sqrt t|y|}}{|y|}
\right)$.
By the Schr\"odinger scaling,
\begin{equation}
G(t\theta,\sqrt t\,y)
=t^{-1/2}L_\theta(y), \quad 
L_\theta(y):=\int_0^\theta e^{\I\alpha\Delta_y}\delta_0(y)\,d\alpha.
\end{equation}
Therefore,
\begin{equation}
t^{-1/2}R_{\mathrm{phase},\delta}^\mu(t)(\sqrt t\,y)
=\int_0^1
\left(m^\mu_{t,\theta}(y)-1\right)
L_\theta(y)
\,d\theta .
\end{equation}
For each fixed $0<\theta\le1$, we have $m^\mu_{t,\theta}(y)\to1 \
\text{for a.e. }y\in\mathbb R^3$.
Moreover, $|m^\mu_{t,\theta}(y)-1|\le2$.
Since $L_\theta\in L^2_y$ by \cref{eq:L-theta-bound} of \cref{lem:free-delta-primitive-bound}, dominated convergence in
$y$ gives
\begin{equation}
\left\|
\left(m^\mu_{t,\theta}-1\right)L_\theta
\right\|_{L^2_y}
\to0
\quad
\text{for each fixed }0<\theta\le1.
\end{equation}
On the other hand, by \cref{eq:L-theta-bound},
\begin{equation}
\left\|
\left(m^\mu_{t,\theta}-1\right)L_\theta
\right\|_{L^2_y}
\le
2\|L_\theta\|_{L^2_y}
\le
C\theta^{1/4}.
\end{equation}
Since $\theta^{1/4}\in L^1(0,1)$, dominated convergence in $\theta$, together with Minkowski's inequality give us
$t^{-1/2}R_{\mathrm{phase},\delta}^\mu(t)(\sqrt t\,\cdot)
\to 0
\quad
\text{in }L^2_y.$
It remains to translate this rescaled convergence back to the original
$L^2_x$ norm. By the change of variables $x=\sqrt t\,y$,
\begin{equation}
\left\|
t^{-1/2}R_{\mathrm{phase},\delta}^\mu(t)(\sqrt t\,\cdot)
\right\|_{L^2_y}
=t^{-5/4}
\|R_{\mathrm{phase},\delta}^\mu(t)\|_{L^2_x}.
\end{equation}
Since the left-hand side tends to $0$, we obtain
$t^{-5/4}
\|R_{\mathrm{phase},\delta}^\mu(t)\|_{L^2_x}
\to0$.
Equivalently,
\begin{equation}
\|R_{\mathrm{phase},\delta}^\mu(t)\|_{L^2_x}
=o(t^{5/4}),
\end{equation}
which proves the original $L^2_x$ statement.
\end{proof}

\begin{lemma}[Replacement of the delta coefficient]
\label[lemma]{lem:replace_delta_coefficient_lower_bound}
Let $\psi_\ast\in H^2(\RR^3)$, and set
\[
g_{s,\tau}:=e^{\I(s-\tau)\Delta}\psi_\ast,
\qquad
0\le \tau\le s\le t.
\]
Then
\begin{equation}
\label{eq:delta-coefficient-replacement-conclusion}
-4\pi
\int_0^t
e^{-\I s V_\mu}
\int_0^s
g_{s,\tau}(0)e^{\I\tau\Delta}\delta_0
\,d\tau\,ds
=
-4\pi\psi_\ast(0)
\int_0^t e^{-\I s V_\mu}G(s,\cdot)\,ds
+o_{L^2_x}(t^{5/4}),
\end{equation}
where $G(s,\cdot)$ is defined in \cref{eq:G}.
\end{lemma}

\begin{proof}
For the delta contribution, we must also justify replacing the
time-dependent coefficient $g_{s,\tau}(0)$ by its limiting value
$\psi_\ast(0)$. This is analogous in spirit to the preceding phase
replacement, where a pointwise in time control on the factor inside the
integral is not enough to identify the leading time-integrated term.
Here the relevant pointwise control is the Sobolev bound $|g_{s,\tau}(0)|\le C\|\psi_\ast\|_{H^2}$, which gives only an upper bound. 
It does not by itself rule out the possibility that the variation of $g_{s,\tau}(0)$ contributes at the same $t^{5/4}$ scale as the leading delta term, and hence could modify or partially cancel the leading contribution.
We therefore isolate
this coefficient variation and show that it is lower order.

By the definition of $g_{s,\tau}$, we have
$g_{s,\tau}(0)-\psi_\ast(0)
=\left(e^{\I(s-\tau)\Delta}\psi_\ast\right)(0)
-\psi_\ast(0)$.
Thus, it suffices to prove
\begin{equation}
\left\|
\int_0^t
e^{-\I s V_\mu}
\int_0^s
\left[
\left(e^{\I(s-\tau)\Delta}\psi_\ast\right)(0)
-\psi_\ast(0)
\right]
e^{\I\tau\Delta}\delta_0
\,d\tau\,ds
\right\|_{L^2}
=o(t^{5/4}).
\label{eq:coeff-delta-remainder-small}
\end{equation}
To prove \cref{eq:coeff-delta-remainder-small}, it suffices to establish
the corresponding estimate for arbitrary $f\in C_c^\infty(\RR^3)$ and
then use the uniform $H^2$ estimate proved below to pass to
$\psi_\ast\in H^2(\RR^3)$ by density. 
Now notice that, for $f\in C_c^\infty(\RR^3)$, the Schr\"odinger group
$e^{\I r\Delta}$ gives a differentiable time-dependent coefficient. More
precisely, set
$q_f(r):=\left(e^{\I r\Delta}f\right)(0)$.
Since $f\in C_c^\infty(\RR^3)$, the map
$r\mapsto e^{\I r\Delta}f$ is continuously differentiable with values
in $H^2(\RR^3)$, and $q_f'(r)
=\I\left(e^{\I r\Delta}\Delta f\right)(0)$.
The strong continuity of the Schrödinger group on $H^2(\RR^3)$,
together with the embedding
$H^2(\RR^3)\hookrightarrow C^0(\RR^3)$, shows that $q_f'$ is
continuous. In particular, $\sup_{0\le r\le1}|q_f'(r)|<\infty$.

Since $q_f(0)=f(0)$, by the fundamental theorem of calculus 
$q_f(s-\tau)-q_f(0)
=\int_0^{s-\tau}q_f'(r)\,dr$.
Thus, by Fubini's theorem, we have
\begin{equation}
\int_0^s
\left(q_f(s-\tau)-q_f(0)\right)
e^{\I\tau\Delta}\delta_0
\,d\tau
=\int_0^s
q_f'(r)
\left(
\int_0^{s-r}e^{\I\tau\Delta}\delta_0\,d\tau
\right)
\,dr.
\end{equation}
Returning to the delta integral, we obtain
\begin{equation}
\begin{aligned}
\left\|
\int_0^s
\left(q_f(s-\tau)-q_f(0)\right)
e^{\I\tau\Delta}\delta_0
\,d\tau
\right\|_{L^2}
&\le
\sup_{0\le r\le1}|q_f'(r)|
\int_0^s
\left\|
\int_0^{s-r}e^{\I\tau\Delta}\delta_0\,d\tau
\right\|_{L^2}
\,dr                                    \\
&\le
C\sup_{0\le r\le1}|q_f'(r)|
\int_0^s(s-r)^{1/4}\,dr            
\le
C_f s^{5/4},
\end{aligned}
\end{equation}
where the first inequality follows from the preceding Fubini identity
and Minkowski's inequality, together with the boundedness of $q_f'$,
while the second follows from
\cref{eq:free-delta-primitive-bound}.
Since multiplication by $e^{-\I s V_\mu}$ is unitary on $L^2$, it
follows that
\begin{equation}\label{eq:small_oh}
\begin{aligned}
&\left\|
\int_0^t
e^{-\I s V_\mu}
\int_0^s
\left[
\left(e^{\I(s-\tau)\Delta}f\right)(0)-f(0)
\right]
e^{\I\tau\Delta}\delta_0
\,d\tau\,ds
\right\|_{L^2}                            \\
&\le
\int_0^t
\left\|
\int_0^s
\left[
\left(e^{\I(s-\tau)\Delta}f\right)(0)-f(0)
\right]
e^{\I\tau\Delta}\delta_0
\,d\tau
\right\|_{L^2}
\,ds                                   \le
C_f\int_0^t s^{5/4}\,ds
\le
C_f t^{9/4}
=o(t^{5/4}).
\end{aligned}
\end{equation}
Here, the first inequality uses the triangle inequality and the
unitarity of multiplication by $e^{-\I s V_\mu}$, while the second
uses the preceding $s^{5/4}$ estimate.

It remains to extend the result to $\psi_\ast\in H^2(\RR^3)$. This
is a standard density argument, which we include for completeness.
For this purpose, we first record the uniform $H^2$ estimate needed
to pass to the limit.
On the one hand, the $H^2\to L^2$ estimate for the delta contribution gives
\begin{equation}
\left\|
\int_0^t
e^{-\I s V_\mu}
\int_0^s
\left(e^{\I(s-\tau)\Delta}h\right)(0)
e^{\I\tau\Delta}\delta_0
\,d\tau\,ds
\right\|_{L^2}
\le
C_\mu t^{5/4}\|h\|_{H^2},
\quad
h\in H^2(\RR^3).
\end{equation}
On the other hand, since $G(s,\cdot)$ is defined in \cref{eq:G}, Sobolev embedding and \cref{eq:free-delta-primitive-bound} yield
\begin{equation}
\left\|
h(0)\int_0^t e^{-\I s V_\mu}G(s,\cdot)\,ds
\right\|_{L^2}
\le
|h(0)|\int_0^t\|G(s,\cdot)\|_{L^2}\,ds    
\le
C\|h\|_{H^2}\int_0^t s^{1/4}\,ds         
\le
C t^{5/4}\|h\|_{H^2}.
\end{equation}
Therefore, subtracting the constant coefficient term and using the
triangle inequality, we obtain
\begin{equation}
\left\|
\int_0^t
e^{-\I s V_\mu}
\int_0^s
\left[
\left(e^{\I(s-\tau)\Delta}h\right)(0)-h(0)
\right]
e^{\I\tau\Delta}\delta_0
\,d\tau\,ds
\right\|_{L^2}
\le
C_\mu t^{5/4}\|h\|_{H^2}.
\label{eq:coeff-delta-uniform-H2-bound}
\end{equation}
In particular, we can choose $f_n\in C_c^\infty(\RR^3)$ such that
$\|f_n-\psi_\ast\|_{H^2}\to0$.
Using \cref{eq:small_oh} with the smooth function $f_n$, and
\cref{eq:coeff-delta-uniform-H2-bound} with $h=\psi_\ast-f_n$, we obtain:
\begin{equation*}
\begin{aligned}
&\limsup_{t\downarrow0}
t^{-5/4}
\left\|
\int_0^t
e^{-\I s V_\mu}
\int_0^s
\left[
\left(e^{\I(s-\tau)\Delta}\psi_\ast\right)(0)
-\psi_\ast(0)
\right]
e^{\I\tau\Delta}\delta_0
\,d\tau\,ds
\right\|_{L^2}                            \\
&\quad\le
\limsup_{t\downarrow0}
t^{-5/4}
\left\|
\int_0^t
e^{-\I s V_\mu}
\int_0^s
\left[
\left(e^{\I(s-\tau)\Delta}f_n\right)(0)-f_n(0)
\right]
e^{\I\tau\Delta}\delta_0
\,d\tau\,ds
\right\|_{L^2}
+C_\mu\|\psi_\ast-f_n\|_{H^2}             \\
&\quad=
C_\mu\|\psi_\ast-f_n\|_{H^2}.
\end{aligned}
\end{equation*}
The last equality follows from the smooth estimate, since $n$ is
fixed when $t\downarrow0$. Letting $n\to\infty$ proves
\cref{eq:coeff-delta-remainder-small}.
Finally, writing
$g_{s,\tau}(0)
=\psi_\ast(0)
+\left[
\left(e^{\I(s-\tau)\Delta}\psi_\ast\right)(0)
-\psi_\ast(0)
\right]$,
and using the definition of $G$ in \cref{eq:G}, we conclude that
\begin{equation}
-4\pi
\int_0^t
e^{-\I s V_\mu}
\int_0^s
g_{s,\tau}(0)e^{\I\tau\Delta}\delta_0
\,d\tau\,ds                          =-4\pi\psi_\ast(0)
\int_0^t e^{-\I s V_\mu}G(s,\cdot)\,ds
+o_{L^2_x}(t^{5/4}).
\end{equation}
This is \cref{eq:delta-coefficient-replacement-conclusion}.
\end{proof}

\section{Proof of \cref{lem:exterior_zero_order}}
\label{app:non_delta}

\begin{proof}
By Minkowski's integral inequality and the unitarity of the free Schr\"odinger group on $L^2(\RR^3)$,
\begin{equation}
\begin{aligned}
\left\|
\int_0^t\int_0^s
e^{i\tau\Delta}
\mu^2 V_\mu e^{i(s-\tau)\Delta}f
\,d\tau\,ds
\right\|_{L^2}  
&\le
\int_0^t\int_0^s
\left\|
e^{i\tau\Delta}
\mu^2 V_\mu e^{i(s-\tau)\Delta}f
\right\|_{L^2}
\,d\tau\,ds  \\
&=
\mu^2
\int_0^t\int_0^s
\left\|
V_\mu e^{i(s-\tau)\Delta}f
\right\|_{L^2}
\,d\tau\,ds .
\end{aligned}
\end{equation}
By H\"older's inequality,
\begin{equation}
\left\|
V_\mu e^{i(s-\tau)\Delta}f
\right\|_{L^2}
\le
\|V_\mu\|_{L^2}
\left\|
e^{i(s-\tau)\Delta}f
\right\|_{L^\infty}.
\end{equation}
By Sobolev embedding, and since $e^{ir\Delta}$ preserves the $H^2$ norm, we have
\begin{equation}
\left\|
e^{i(s-\tau)\Delta}f
\right\|_{L^\infty}
\le
C\left\|
e^{i(s-\tau)\Delta}f
\right\|_{H^2}
=C\|f\|_{H^2}.
\end{equation}
Moreover,
\begin{equation}
\|V_\mu\|_{L^2}^2
=\int_{\mathbb R^3}\frac{e^{-2\mu |x|}}{|x|^2}\,dx
=4\pi\int_0^\infty e^{-2\mu r}\,dr
=\frac{2\pi}{\mu}.
\end{equation}
Therefore, for all $0\le \tau\le s\le t$,
\begin{equation}
\left\|
e^{i\tau\Delta}
\mu^2 V_\mu e^{i(s-\tau)\Delta}f
\right\|_{L^2}
\le C_\mu \|f\|_{H^2}.
\end{equation}
It follows that
\begin{equation}
\left\|
\int_0^t\int_0^s
e^{i\tau\Delta}
\mu^2 V_\mu e^{i(s-\tau)\Delta}f
\,d\tau\,ds
\right\|_{L^2}
\le C_\mu \|f\|_{H^2}
\int_0^t\int_0^s d\tau\,ds  
\le C_\mu t^2\|f\|_{H^2}.
\end{equation}

\end{proof}

\end{document}